\documentclass[12pt, reqno]{amsart}

\usepackage{amsthm, amsmath, amsfonts, amssymb, graphicx, tikz-cd, comment, array, hyperref, extarrows, mathtools, todonotes, calc}
\usepackage[capitalise]{cleveref}
\usepackage{enumitem}
\usepackage{float} 
\usepackage[justification=centering]{caption} 

\newtheorem{theorem}{Theorem}[section]
\newtheorem{lemma}[theorem]{Lemma}
\newtheorem{proposition}[theorem]{Proposition}

\newtheorem{claim}[theorem]{Claim}

\theoremstyle{definition}
\newtheorem{definition}[theorem]{Definition}

\newtheorem{remark}[theorem]{Remark}

\def\KHaus{{\mathsf{KHaus}}}

\def\Clop{{\mathsf{Clop}}}
\def\RO{{\mathcal{RO}}}

\def\G{{\mathcal{G}}}
\def\int{{\mathsf{int}}}
\def\cl{{\mathsf{cl}}}

\def\dev{{\mathsf{DeV}}}

\def\R{{\mathfrak{R}}}

\def\KRFrm{{\mathsf{KRFrm}}}
\def\KRFrmR{{\mathsf{KRFrm^P}}}
\def\KRFrmC{{\mathsf{KRFrm^C}}}

\def\O{{\mathcal{O}}}

\def\Ult{{\mathsf{Ult}}}

\def\R{\mathrel{R}}
\def\S{\mathrel{S}}
\def\T{\mathrel{T}}
\def\Q{\mathrel{Q}}
\def\E{\mathrel{E}}

\def\RI{{\mathcal{RI}}}

\def\NI{{\mathcal{NI}}}

\newcommand{\KHausR}{\mathsf{KHaus}^\mathsf{R}}
\newcommand{\KHausC}{\mathsf{KHaus}^\mathsf{C}}

\newcommand{\Gle}{\mathsf{Gle}}
\newcommand{\GleR}{\mathsf{Gle}^\mathsf{R}}
\newcommand{\GleC}{\mathsf{Gle}^\mathsf{C}}

\newcommand{\StoneER}{\mathsf{StoneE}^\mathsf{R}}
\newcommand{\StoneEC}{\mathsf{StoneE}^\mathsf{C}}
\newcommand{\StoneEF}{\mathsf{StoneE}^\mathsf{F}}

\newcommand{\SubSfive}{\mathsf{SubS5^S}}
\newcommand{\SubSfiveCS}{\mathsf{SubS5^{CS}}}
\newcommand{\SubSfiveF}{\mathsf{SubS5^F}}

\newcommand{\devC}{\mathsf{DeV^C}}
\newcommand{\devS}{\mathsf{DeV^S}}
\newcommand{\devCS}{\mathsf{DeV^{CS}}}
\newcommand{\devF}{\mathsf{DeV^F}}

\newcommand{\StoneEf}{\mathsf{StoneE^F}}

\newcommand{\KRFrmP}{{\mathsf{KRFrm^P}}}
\newcommand{\I}{{\mathcal{I}}}
\newcommand{\bool}{{\mathfrak{B}}}

\newcommand{\F}{\Delta}

\newcommand\conv[1]{{#1 \raisebox{-2pt}{\scalebox{1.2}{$\breve{\hspace{3.5pt}}$}}}}
\newcommand\inv[1]{\widehat{#1}}

\renewcommand{\c}{\mathsf{c}}

\title{MacNeille completions of subordination algebras}

\author[M. Abbadini]{Marco Abbadini}
\address{School of Computer Science,
University of Birmingham,
B15 2TT Birmingham (UK)}
\email{marco.abbadini.uni@gmail.com}

\author[G. Bezhanishvili]{Guram Bezhanishvili}
\address{Department of Mathematical Sciences\\
New Mexico State University\\
Las Cruces NM 88003\\
USA}
\email{guram@nmsu.edu}

\author[L. Carai]{Luca Carai}
\address{Dipartimento di Matematica,
Universit\`a degli Studi di Milano,
20133 Milan (MI),
Italy}
\email{luca.carai.uni@gmail.com}

\date{}


\begin{document}
	
\begin{abstract}
$\mathsf{S5}$-subordination algebras are a natural generalization of de Vries algebras. Recently it was proved that the category $\SubSfive$ of $\mathsf{S5}$-su\-bo\-rdi\-na\-tion algebras and compatible subordination relations between them is equivalent to the category of compact Hausdorff spaces and closed relations. We generalize MacNeille completions of boolean algebras to the setting of $\mathsf{S5}$-subordi\-nation algebras, and utilize the relational nature of the morphisms in $\SubSfive$ to prove that the MacNeille completion functor establishes an equivalence between $\SubSfive$ and its full subcategory consisting of de Vries algebras. 
We also show that the functor that associates to each $\mathsf{S5}$-subordination algebra the frame of its round ideals establishes a dual equivalence between $\SubSfive$ and the category of compact regular frames and preframe homomorphisms. Our results are choice-free and provide further insight into Stone-like dualities for compact Hausdorff spaces with various morphisms between them. In particular, we show how they restrict to the wide subcategories of $\SubSfive$ corresponding to continuous relations and continuous functions between compact Hausdorff spaces.
\end{abstract}
	
\maketitle

\tableofcontents

\section{Introduction}

With each compact Hausdorff space $X$, we can associate numerous algebraic structures that determine $X$ up to homeomorphism. This 
yields various dualities for the category $\KHaus$ of compact Hausdorff spaces and continuous functions. In this paper we are interested in two dualities for $\KHaus$ from pointfree topology. By Isbell duality \cite{Isb72}, $\KHaus$ is dually equivalent to the category $\KRFrm$ of compact regular frames and frame homomorphisms; and by de Vries duality \cite{deV62}, $\KHaus$ is dually equivalent to the category $\dev$ of de Vries algebras and de Vries morphisms. 

Isbell duality is established by working with the contravariant functor $\O \colon \KHaus\to\KRFrm$ which associates with each compact Hausdorff space $X$ the compact regular frame $\O(X)$ of open subsets of $X$, and with each continuous function $f\colon X\to Y$ the frame homomorphism $f^{-1} \colon \O(Y)\to\O(X)$. De Vries duality is established by working with the contravariant functor $\RO\colon \KHaus\to\dev$. Writing $\int$ for the interior and $\cl$ for the closure, $\RO$ associates with each $X\in\KHaus$ the de Vries algebra $(\RO(X),\prec)$ of regular open subsets of $X$, where $U\prec V$ iff $\cl(U)\subseteq V$, and with each continuous function $f\colon X\to Y$ the de Vries morphism $\RO(f)\colon \RO(Y)\to\RO(X)$ given by $\RO(f)(V)=\int(\cl f^{-1}[V])$ for each $V\in\RO(Y)$.

As a consequence of Isbell and de Vries dualities, $\KRFrm$ is equivalent to $\dev$. This equivalence can be obtained directly, without first passing to $\KHaus$ \cite{Bez12}. We thus arrive at the following diagram, where the horizontal arrow represents an equivalence and the slanted arrows with the letter $d$ on top represent dual equivalences.
\begin{figure}[!h]
\[
\begin{tikzcd}
 & \KHaus \arrow[dl, leftrightarrow, "d"'] \arrow[dr, leftrightarrow, "d"] & \\
\KRFrm \arrow[rr, leftrightarrow] & & \dev
\end{tikzcd}
\]
\end{figure}

Several authors have considered generalizations of $\KHaus$ where functions are replaced by relations. A relation $R$ between two compact Hausdorff spaces $X$ and $Y$ is {\em closed} if $R$ is a closed subset of $X\times Y$ and it is {\em continuous} if in addition the $R$-preimage of each open subset of $Y$ is open in $X$. A function between compact Hausdorff spaces is closed iff it is continuous. But for relations this results in two different categories $\KHausR$ and $\KHausC$. In the former, morphisms are closed relations; and in the latter, they are continuous relations. Clearly $\KHaus$ is a wide subcategory of $\KHausC$, which in turn is a wide subcategory of $\KHausR$. 

In \cite{BezhanishviliGabelaiaEtAl2019} $\KRFrm$ was generalized to $\KRFrmC$, $\dev$ to $\devC$ (see \cref{sec: prelims} for the definitions of these categories), and it was shown that the commutative diagram above extends to the following commutative diagram.
\begin{figure}[!h]
\[
\begin{tikzcd}
 & \KHausC \arrow[dl, leftrightarrow, "d"'] \arrow[dr, leftrightarrow, "d"] & \\
\KRFrmC \arrow[rr, leftrightarrow] & & \devC
\end{tikzcd}
\]
\end{figure}

On the other hand, in \cite{Tow96, JKM01} $\KRFrm$ was generalized to $\KRFrmR$, where morphisms are preframe homomorphisms (that is, they preserve finite meets and directed joins), and it was shown that $\KRFrmR$ is dually equivalent to $\KHausR$. In a recent paper \cite{ABC22a} we introduced the category $\devS$ whose objects are de Vries algebras and whose morphisms are compatible subordination relations. We proved that $\devS$ is equivalent to $\KHausR$ and hence dually equivalent to $\KRFrmR$. Thus, we arrive at the following commutative diagram that extends the two diagrams above.
\begin{figure}[!h]
\[
\begin{tikzcd}
 & \KHausR \arrow[dl, leftrightarrow, "d"'] \arrow[dr, leftrightarrow] & \\
\KRFrmP \arrow[rr, leftrightarrow, "d"] & & \devS
\end{tikzcd}
\]
\end{figure}

Our aim here is to give a direct choice-free proof of the duality between $\KRFrmR$ and $\devS$. From this we derive a direct choice-free proof of the equivalence between $\KRFrmC$ and $\devC$, as well as an alternative choice-free proof of the equivalence between $\KRFrm$ and $\dev$. 

Our main tool is the category $\SubSfive$ of $\mathsf{S5}$-subordination algebras introduced in \cite{ABC22a}.
Objects of $\SubSfive$ were already considered by Meenakshi \cite{Mee66}, who studied proximity relations on an arbitrary boolean algebra.
In \cite{ABC22a} we used a generalization of Stone duality to closed relations \cite{Cel18,KurzMoshierEtAl2023} and the machinery of allegories \cite{FS90} to show that $\SubSfive$ is equivalent to the category $\StoneER$ whose objects are Stone spaces equipped with a closed equivalence relation and whose morphisms are special closed relations (see \cref{def-StoneER and GleR}\eqref{def-StoneER}). Since $\devS$ is a full subcategory of $\SubSfive$, restricting this equivalence yields an equivalence between $\devS$ and the full subcategory $\GleR$ of $\StoneER$ consisting of Gleason spaces. 
It turns out that these four categories are equivalent to $\KHausR$. Consequently, $\devS$ is equivalent to $\SubSfive$, but the proof goes through $\KHausR$ and hence uses the axiom of choice.

In this paper we generalize MacNeille completions of boolean algebras to $\mathsf{S5}$-subordination algebras and give a direct choice-free proof of the equivalence between $\SubSfive$ and $\devS$.
We also specialize the notion of a round ideal of a proximity lattice \cite{War75} to our setting to obtain a contravariant functor from $\SubSfive$ to $\KRFrmP$, yielding a choice-free proof that $\SubSfive$ is dually equivalent to $\KRFrmR$.
We thus arrive at the following commutative diagram.
\begin{figure}[!h]
\[
\begin{tikzcd}
 & \SubSfive \arrow[dl, leftrightarrow, "d"'] \arrow[dr, leftrightarrow] & \\
\KRFrmP \arrow[rr, leftrightarrow, "d"] & & \devS
\end{tikzcd}
\]
\end{figure}

We also study the wide subcategories of these categories whose morphisms encode continuous relations and continuous functions between compact Hausdorff spaces.

The paper is organized as follows.
In \cref{sec: prelims} we recall the existing dualities for compact Hausdorff spaces that are relevant for our purposes.
In \cref{sec: id comple} we describe the round ideal functor from $\SubSfive$ to $\KRFrmP$.
In \cref{sec:MacNeille} we define MacNeille completions of $\mathsf{S5}$-subordination algebras and prove that the resulting functor yields an equivalence between $\SubSfive$ and $\devS$. We then use this result to show that the round ideal functor from $\SubSfive$ to $\KRFrmP$ is a dual equivalence.
In \cref{sec: cont subs} we study the wide subcategories of these categories whose morphisms encode continuous relations between compact Hausdorff spaces. In \cref{sec:functional} we further restrict our attention to the morphisms that encode continuous functions between compact Hausdorff spaces. Finally, in \cref{sec: dually} we give dual descriptions of the round ideal and MacNeille completions of $\mathsf{S5}$-subordination algebras.

All the categories considered in this paper are listed in \cref{table1,table2,table3,table4} and all the equivalences and dual equivalences in \cref{fig:diagram2} at the end of \cref{sec:functional}.

\section{Preliminaries} \label{sec: prelims}

In this section we briefly recall Isbell duality, de Vries duality, and their generalizations.
We start by recalling some basic definitions from pointfree topology (see, e.g., \cite{PP12}). A {\em frame} or {\em locale} is a complete lattice $L$ satisfying the join-infinite distributive law 
\[
a\wedge\bigvee S=\bigvee\{ a\wedge s \mid s\in S \}.
\]
Each $a\in L$ has the {\em pseudocomplement} given by $a^*=\bigvee\{ x\in L \mid a\wedge x=0 \}$. We say that $a$ is {\em compact} if $a\le \bigvee S$ implies $a\le \bigvee T$ for some finite $T\subseteq S$, and that $a$ is {\em well-inside} $b$ (written $a\prec b$) if $a^*\vee b=1$. A frame $L$ is {\em compact} if $1$ is compact and it is {\em regular} if $a=\bigvee\{ x\in L \mid x\prec a\}$ for each $a\in L$.

A {\em frame homomorphism} between two frames is a map 
that preserves arbitrary joins and finite meets.
\label{def-of-KRFrm}\label{def-of-KHaus}We recall from the introduction that $\KRFrm$ is the category of compact regular frames and frame homomorphisms and that $\KHaus$ is the category of compact Hausdorff spaces and continuous functions.
\begin{theorem} [Isbell duality] \label{t:Isbell}
	$\KRFrm$ is dually equivalent to $\KHaus$. 
\end{theorem}

A \emph{preframe homomorphism} between two frames is a map that preserves directed joins and finite meets.
\label{def-of-KRFrmP} We let $\KRFrmP$ be the category of compact regular frames and preframe homomorphisms. Clearly $\KRFrm$ is a wide subcategory of $\KRFrmP$. 

We recall that a relation $R\subseteq X\times Y$ between compact Hausdorff spaces is {\em closed} if $R$ is a closed subset of $X\times Y$. As usual, for $x\in X$ and $y\in Y$, we write
\[
R[x]=\{y\in Y\mid x \R y\} \quad \mbox{and} \quad R^{-1}[y]=\{x\in X\mid x \R y\}.
\]
Also, for $F\subseteq X$ and $G\subseteq Y$, we write
\[
R[F]=\bigcup\{ R[x] \mid x\in F\} \quad \mbox{and} \quad R^{-1}[G]=\bigcup \{ R^{-1}[y] \mid y\in G\}.
\]
Then $R$ is closed iff $R[F]$ is closed for each closed $F\subseteq X$ and $R^{-1}[G]$ is closed for each closed $G\subseteq Y$ (see, e.g., \cite[Lem.~2.12]{BBSV17}).
\label{def-KHausR} We let $\KHausR$ be the category of compact Hausdorff spaces and closed relations, where identities are identity relations and composition is relation composition. We recall that for two relations $R_1 \subseteq X_1\times X_2$ and $R_2\subseteq X_2\times X_3$ the relation composition $R_2 \circ R_1 \subseteq X_1 \times X_3$ is defined by 
\[
x_1 \mathrel{(R_2 \circ R_1)} x_3 \iff \exists x_2 \in X_2 : x_1 \R_1 x_2\text{ and }x_2 \R_2 x_3.
\] 

The category $\KHausR$ is a full subcategory of the category of stably compact spaces and closed relations introduced and studied in \cite{JKM01}.
It is symmetric in that if $R$ is a closed relation, then its converse $\conv{R}\colon X_2\to X_1$ (defined by $y \mathrel{\conv{R}} x$ iff $x \R y$) is also closed. This defines a dagger on $\KHausR$ with which $\KHausR$ forms an allegory (see, e.g., \cite[Lem.~3.6]{ABC22a}).
The following theorem generalizes Isbell duality: 

\begin{theorem} [\cite{Tow96, JKM01}] \label{thm:isbell-gen}
$\KRFrmP$ is dually equivalent to $\KHausR$.
\end{theorem}

\label{def-KHausC} A closed relation $R\subseteq X\times Y$ between compact Hausdorff spaces is {\em continuous} if $V$ open in $Y$ implies $R^{-1}[V]$ is open in $X$. Let $\KHausC$ be the wide subcategory of $\KHausR$ whose morphisms are continuous relations. 

\label{def-c-morphism} In \cite[Def.~4.3]{BezhanishviliGabelaiaEtAl2019}, motivated by Johnstone's construction of the Vietoris frame of a compact regular frame \cite[Sec.~III.4]{Joh82}, a preframe homomorphism $\Box \colon L \to M$ between compact regular frames is called {\em continuous} or a \emph{$\mathsf{c}$-morphism} if there is a join-preserving $\Diamond \colon L \to M$ such that 
\[
\Box(a \vee b) \le \Box a \vee \Diamond b \quad \mbox{and} \quad \Box a \wedge \Diamond b \le \Diamond(a \wedge b).
\] 
Let $\KRFrmC$ be the wide subcategory of $\KRFrmR$ whose morphisms are $\c$-morphisms.
The duality of \cref{thm:isbell-gen} then restricts to the following generalization of Isbell duality:

\begin{theorem}[{\cite[Thm.~4.8]{BezhanishviliGabelaiaEtAl2019}}] \label{thm:KRFrmC-dual-KHausC}
	$\KRFrmC$ is dually equivalent to $\KHausC$.
\end{theorem}

Letting $\Diamond = \Box$, we can identify $\KRFrm$ with a wide subcategory of $\KRFrmC$. Thus, we arrive at the following diagram, where the hook arrows represent inclusions of wide subcategories and the horizontal arrows dual equivalences.
\[
\begin{tikzcd}[column sep=4em]
	\KRFrmP \arrow[r, leftrightarrow, "d"] & \KHausR \\
	\KRFrmC \arrow[u, hookrightarrow] \arrow[r, leftrightarrow, "d"] & \KHausC \arrow[u, hookrightarrow] \\
	\KRFrm \arrow[u, hookrightarrow] \arrow[r, leftrightarrow, "d"] & \KHaus \arrow[u, hookrightarrow]
\end{tikzcd}
\]

\begin{definition} \cite[Def.~2.4]{ABC22a}
Let $A,B$ be boolean algebras.
A relation $S \subseteq A \times B$ is a \emph{subordination} if $S$ satisfies the following conditions for all $a,b \in A$ and $c,d\in B$:
\begin{enumerate}[label = (S\arabic*), ref = S\arabic*]
\item \label{S1} $0 \S 0$ and $1 \S 1$;
\item \label{S2} $a,b \S c$ implies $(a\vee b) \S c$; 
\item \label{S3} $a \S c,d$ implies $a \S (c\wedge d)$;  
\item \label{S4} $a\le b \S c\le d$ implies $a \S d$. 
\end{enumerate}
\end{definition}

\begin{remark}
The axioms \eqref{S1}--\eqref{S4} are equivalent to saying that $S$ is a bounded sublattice of $A \times B$ satisfying \eqref{S4}.
\end{remark}

When $A=B$, we say that $S$ is a \emph{subordination on} $A$. These were introduced in~\cite{BBSV17} as a counterpart of quasi-modal operators~\cite{Celani2001} and precontact relations~\cite{DV06,DuentschVakarelov2007}. As follows from \cite[Thm.~2.22]{BBSV17}, subordinations on $A$ correspond to closed relations $R$ on the Stone space of $A$. By \cite{Celani2001,DuentschVakarelov2007} (see also \cite[Lem.~4.6]{BBSV17}),  we can characterize reflexivity, symmetry, and transitivity of $R$ by the following axioms, where we write $\neg a$ for the complement of $a$ in $A$.

\begin{enumerate}[label = (S\arabic*), ref = S\arabic*]
\setcounter{enumi}{4}
\item \label{S5} $a \S b$ implies $a \le b$;
\item \label{S6} $a \S b$ implies $\neg b \S \neg a$;
\item \label{S7} $a \S b$ implies there is $c \in A$ with $a \S c$ and $c \S b$.
\end{enumerate}

Following the modal logic nomenclature, the pairs $(B,S)$ where $B$ is a boolean algebra and $S$ is a subordination on $B$ satisfying \eqref{S5}--\eqref{S7} were called \emph{$\mathsf{S5}$-subordination} algebras in \cite{ABC22a}.

These algebras were first introduced in \cite{Mee66}, where the notion of a proximity on a set was generalized to an arbitrary boolean algebra. 
Further generalizations include proximity lattices \cite{War75,Smy92}, proximity algebras \cite{GK81}, and proximity frames \cite{BH14}. We point out that $\mathsf{S5}$-subordination algebras are exactly the proximity algebras of \cite{GK81} where the underlying Heyting algebra is a boolean algebra.

\begin{definition} 
Let $\mathbf{B}=(B, S)$ be an $\mathsf{S5}$-subordination algebra.
\begin{enumerate}
\item\cite[Def.~1.1.1]{deV62} We call $\mathbf{B}$ a \emph{compingent algebra} if $S$ satisfies the following axiom:
	\begin{enumerate}[label = (S\arabic*), start= 8, ref = S\arabic*]
		\item \label{S8} If $a \neq 0$, then there is $b \neq 0$ with $b \S a$.
	\end{enumerate}
\item\cite[Def.~3.2]{Bez10} We call $\mathbf{B}$ a \emph{de Vries algebra} if $\mathbf{B}$ is a compingent algebra and $B$ is a complete boolean algebra. 
\end{enumerate}
\end{definition}

\begin{remark}
As was pointed out in \cite[Prop.~7.4]{BH14}, de Vries algebras are exactly those proximity frames where the frame is boolean. 
\end{remark}

A {\em de Vries morphism} between de Vries algebras is a map $f\colon B_1 \to B_2$ satisfying the following conditions:
	\begin{enumerate}[label = (M\arabic*), ref = M\arabic*]
		\item\label{M1} $f(0) = 0$;
		\item\label{M2} $f(a \land b) = f(a) \land f(b)$;
		\item\label{M3} $a \S_1 b$ implies $\lnot f(\lnot a) \S_2 f(b)$;
		\item\label{M4} $f(a) = \bigvee \{f(b) \mid b \S_1 a\}$.
	\end{enumerate}
	The composition of two de Vries morphisms $f \colon B_1 \to B_2$ and $g \colon B_2 \to B_3$ is the de Vries morphism $g * f \colon B_1 \to B_3$ given by
	\[
		(g * f)(a) = \bigvee \{ gf(b) \mid b \S_1 a \}
	\]
	for each $a\in B_1$. \label{def-dev} Let $\dev$ be the category of de Vries algebras and de Vries morphisms, where identity morphisms are identity functions and composition is defined as above.

\begin{theorem} [de Vries duality] \label{t:de-Vries-duality}
	$\dev$ is dually equivalent to $\KHaus$. 
\end{theorem}

In \cite{BezhanishviliGabelaiaEtAl2019} de Vries duality was generalized to a duality for $\KHausC$. For this, the notion of a de Vries additive map from \cite{BBH15} was utilized. We will instead work with the equivalent notion of a de Vries multiplicative map.

\begin{definition}
A map $\Box \colon B_1\to B_2$ between de Vries algebras is {\em de Vries multiplicative} if $\Box 1 = 1$ and for all $a,b,c,d\in B_1$, we have
\[
a \S_1 b \mbox{ and } c \S_1 d \mbox{ imply } (\Box a \wedge \Box c) \S_2 \Box(b \wedge d). 
\]
We call $\Box$ {\em lower continuous} if in addition
\[
\Box a = \bigvee \{ \Box b \mid b \S_1 a\}
\]
for each $a\in B_1$. The composition of two such maps $\Box_1$ and $\Box_2$ is given by
\[
	(\Box_2 * \Box_1) a = \bigvee \{ \Box_2\Box_1 b \mid b \S_1 a \}.
\]
\label{def-devC}Let $\devC$ be the category of de Vries algebras and lower continuous de Vries multiplicative maps, where identity morphisms are identity functions and composition is defined as above. 
\end{definition} 

\begin{remark}\label{rem:box diamond upper lower}
\hfill\begin{enumerate}
\item The results of \cite{BezhanishviliGabelaiaEtAl2019} are stated using de Vries additive maps  that are lower continuous, where we recall that $\Diamond\colon B_1\to B_2$ is de Vries additive if
$\Diamond 0 = 0$ and $a \S_1 b$ and $c \S_1 d$ imply $\Diamond (a \vee c) \S_2 (\Diamond b \vee \Diamond d)$
for all $a,b,c,d\in B_1$, and it
is lower continuous if
$\Diamond a = \bigvee \{ \Diamond b \mid b \S_1 a\}$
for all $a\in B_1$.
To simplify proofs (see, e.g., \cref{lem:continuity BoxT and TBox}), we will work with $\Box$ instead of $\Diamond$.
\item\label{rem:box diamond upper lower:item2}
As observed in~\cite[Rem.~4.11]{BezhanishviliGabelaiaEtAl2019}, working with lower continuous de Vries additive maps is equivalent to working with de Vries multiplicative maps that are upper continuous, i.e.\ maps $\Box$ that satisfy $\Box a = \bigwedge \{ \Box b \mid a \S b\}$. Analogously, working with de Vries multiplicative lower continuous maps is equivalent to working with de Vries additive maps that are upper continuous.
\item 
By a slight adjustment of the proofs of \cite[Thms.~4.21, 4.22]{BBH15} it is not difficult to show that the category of de Vries algebras and de Vries additive upper continuous maps between them is equivalent to the category of de Vries algebras and de Vries additive lower continuous maps between them.
Similarly, one can show that $\devC$ is equivalent to the category of de Vries algebras and upper continuous de Vries multiplicative maps between them, and hence to the category of de Vries algebras and lower continuous de Vries additive maps between them. Thus, the results of~\cite{BezhanishviliGabelaiaEtAl2019} apply to our setting. 
\end{enumerate}
\end{remark}

\begin{theorem}[{\cite[Thm.~4.14]{BezhanishviliGabelaiaEtAl2019}}] \label{t:continuous}
$\devC$ is dually equivalent to $\KHausC$.
\end{theorem}

In \cite{BezhanishviliGabelaiaEtAl2019} obtaining a de Vries like duality for $\KHausR$ was left open. This question was resolved in \cite{ABC22a} by working with special subordination relations between de Vries algebras.
To introduce them, we require the following definition of compatibility.

\begin{definition}\label{def:compatible rel}
For $i=1,2$ let $R_i$ be a binary relation on a set $X_i$. We call a relation $T \colon X_1 \to X_2$ \emph{compatible} if 
$R_2\circ T=T=T\circ R_1$. 
\[
	\begin{tikzcd}
		X_1 \arrow{r}{T} \arrow{rd}{T} \arrow[swap]{d}{R_1} & X_2 \arrow{d}{R_2}\\
		X_1 \arrow[swap]{r}{T} & X_2
	\end{tikzcd}
\]
\end{definition}

Let $\SubSfive$ be the category of $\mathsf{S5}$-subordination algebras and compatible subordinations between them, where the composition of morphisms is the usual composition of relations, and the identity morphism on an $\mathsf{S5}$-subordination algebra $(B,S)$ is the relation $S$. Let $\devS$ be the full subcategory of $\SubSfive$ consisting of de Vries algebras.

To connect $\KHausR$ with $\SubSfive$, it is convenient to first obtain a Stone-like representation of $\mathsf{S5}$-subordination algebras.

\begin{definition}\label{def-StoneER and GleR}

\mbox{}\begin{enumerate}
		\item \label{def-StoneER}
		An \emph{$\mathsf{S5}$-subordination space} is a pair $(X, E)$ where $X$ is a Stone space and $E$ is a closed equivalence relation on $X$.
		We let $\StoneER$ be the category whose objects are $\mathsf{S5}$-subordination spaces and whose morphisms are compatible closed relations between them.
		
		\item	\label{def-GleR}
		A \emph{Gleason space} is an $\mathsf{S5}$-subordination space $(X,E)$ such that $X$ is \emph{extremally disconnected} (i.e., the closure of an open set is open) and $E$ is \emph{irreducible} (i.e., if $F$ is a proper closed subset of $X$, then so is $E[F]$).
		We let $\GleR$ 
		be the full subcategory of $\StoneER$ whose objects are Gleason spaces.
	\end{enumerate}
\end{definition}

\begin{theorem} [{\cite[Cors.~3.14, 4.7]{ABC22a}}] \label{thm:equiv-ABC}
	$\KHausR$, $\StoneER$, $\GleR$, $\SubSfive$, and $\devS$ are equivalent categories.
	\[
	\begin{tikzcd}[row sep=10pt]
		& \StoneER \arrow[r, leftrightarrow]  & \SubSfive \\
		\KHausR \arrow[ur, leftrightarrow] \arrow[dr, leftrightarrow] & & \\
		& \GleR \arrow[uu, hookrightarrow] \arrow[r, leftrightarrow] & \devS \arrow[uu, hookrightarrow]
	\end{tikzcd}
	\]
\end{theorem}

To make the paper self-contained, we briefly describe the functors yielding some of the equivalences of \cref{thm:equiv-ABC}. 

\begin{remark}\label{rem:functors equivalences}
\mbox{}
\begin{enumerate}
	\item \label{i:functor-Q}
	The functor $\mathcal{Q} \colon \StoneER \to \KHausR$ maps an object $(X, E)$ to the quotient space $X/E$, and a morphism $R \colon (X_1, E_1) \to (X_2, E_2)$ to the morphism $\mathcal{Q}(R) \colon \mathcal{Q}(X_1,E_1) \to \mathcal{Q}(X_2,E_2)$ 
	given by $[x]_{E_1} \mathrel{\mathcal{Q}(R)} [y]_{E_2}$ iff $x \mathrel{R} y$ (i.e., $\mathcal{Q}(R) = \pi_2 \circ R \circ \conv{\pi_1}$, where $\pi_1$ and $\pi_2$ are the quotient maps).
	\[
	\begin{tikzcd}
		X_1 \arrow{r}{R} \arrow[swap]{d}{\pi_1} & X_2 \arrow{d}{\pi_2}\\
		X_1/E_1 \arrow[swap]{r}{\mathcal{Q}(R)} & X_2/E_2
	\end{tikzcd}
\]
	\item \label{i:functor-G}
	A quasi-inverse of $\mathcal{Q}$ is given by the Gleason cover functor $\G \colon \KHausR \to \StoneER$ which associates to each compact Hausdorff space $X$ the pair $\G(X)=(\widehat{X}, E)$ where $g \colon \widehat{X} \to X$ is the Gleason cover of $X$ and $x \E y$ iff $g(x)=g(y)$ (for Gleason covers see, e.g., \cite[Sec.~III.3.10]{Joh82}). 
	It also maps a closed relation $R \colon X_1 \to X_2$ to the relation 
	$\G(R) \colon \G(X_1) \to \G(X_2)$ given by $x \mathrel{\G(R)} y$ iff $g_1(x) \mathrel{R} g_2(y)$ (i.e., $\G(R) = \conv{g_2} \circ R \circ g_1$).
		\[
		\begin{tikzcd}
			\widehat{X_1} \arrow[swap]{d}{g_1} \arrow{r}{\G(R)} & \widehat{X_2} \arrow{d}{g_2}\\
			X_1 \arrow[swap]{r}{R}  & X_2 
		\end{tikzcd}
		\]		
	\item The functor $\mathcal{G}$ is also a quasi-inverse of the restriction of the functor $\mathcal{Q}$ to $\GleR$.
	\item\label{i:GleR StoneER equivalence} The inclusion of $\GleR$ into $\StoneER$ is an equivalence whose quasi-inverse is the composition $\G \circ \mathcal{Q}$.
	\item \label{i:functor-Clop}
	The functor $\Clop \colon \StoneER \to \SubSfive$ maps an object $(X, E)$ to $(B, S_E)$, where $B$ is the boolean algebra of clopen subsets of $X$ and $S_E$ is the binary relation on $B$ given by $U \S_E V$ iff $E[U] \subseteq V$. Also, $\Clop$ maps a morphism $R \colon (X_1, E_1) \to (X_2, E_2)$ to the compatible subordination relation $S_R \colon \Clop(X_1,E_1) \to \Clop(X_2,E_2)$ given by $U \S_R V$ iff $R[U] \subseteq V$.
	\item \label{i:functor-Ult}
	A quasi-inverse of $\Clop$ is given by the ultrafilter functor $\Ult \colon \SubSfive \to \StoneER$ which associates to each 
	object $(B,S)$ the pair $\Ult(B,S)=(X,R_S)$ where $X$ is the Stone space of ultrafilters of $B$ and $x \R_S y$ iff $S[x] \subseteq y$. We call $(X,R_S)$ the \emph{$\mathsf{S5}$-subordination space} of $(B,S)$. A morphism $T \colon (B_1,S_1) \to (B_2,S_2)$ is mapped by $\Ult$ to the morphism $R_T \colon \Ult(B_1,S_1) \to \Ult(B_2,S_2)$ given by $x \R_T y$ iff $T[x] \subseteq y$. 
		\item The restrictions $\Clop \colon \GleR \to \devS$ and $\Ult \colon \devS \to \GleR$ are also quasi-inverses of each other.
\end{enumerate}
\end{remark}

It follows from \cref{thm:isbell-gen,thm:equiv-ABC} that $\SubSfive$ is dually equivalent to $\KRFrmP$ and equivalent to $\devS$. 
The main contribution of this paper is to give direct choice-free proofs of these results by generalizing ideal and MacNeille completions of boolean algebras to the setting of $\mathsf{S5}$-subordination algebras, to fill in the empty boxes of the following diagram, and to show that it commutes up to natural isomorphism.
The unlabeled horizontal arrows in the diagram represent equivalences of categories while the ones labeled with the letter $d$ represent dual equivalences. The vertical arrows are inclusions of wide subcategories.

\begin{figure}[h]
\[
\begin{tikzcd}[column sep=0.90cm]
\SubSfive \arrow[r, leftrightarrow] & \devS \arrow[rr, leftrightarrow, "d"] & & \KRFrmP \arrow[r, leftrightarrow, "d"] & \KHausR \arrow[r, leftrightarrow] & \StoneER \arrow[r, leftrightarrow] & \GleR \\
\boxed{\phantom{\dev}} \arrow[u, hookrightarrow] \arrow[r, leftrightarrow] & \boxed{\phantom{\dev}} \arrow[u, hookrightarrow] \arrow[r, leftrightarrow, "d"] & \devC \arrow[r, leftrightarrow] & \KRFrmC \arrow[u, hookrightarrow] \arrow[r, leftrightarrow, "d"] & \KHausC \arrow[u, hookrightarrow] \arrow[r, leftrightarrow] & \boxed{\phantom{\dev}} \arrow[u, hookrightarrow] \arrow[r, leftrightarrow] & \boxed{\phantom{\dev}} \arrow[u, hookrightarrow] \\
\boxed{\phantom{\dev}} \arrow[u, hookrightarrow] \arrow[r, leftrightarrow] & \boxed{\phantom{\dev}} \arrow[u, hookrightarrow] \arrow[r, leftrightarrow, "d"] & \dev \arrow[u, hookrightarrow] \arrow[r, leftrightarrow] & \KRFrm \arrow[u, hookrightarrow] \arrow[r, leftrightarrow, "d"] & \KHaus \arrow[u, hookrightarrow] \arrow[r, leftrightarrow] & \boxed{\phantom{\dev}} \arrow[u, hookrightarrow] \arrow[r, leftrightarrow] & \boxed{\phantom{\dev}} \arrow[u, hookrightarrow]
\end{tikzcd}
\]
\caption{}\label{fig:diagram1}
\end{figure}

\section{Round ideals of $\mathsf{S5}$-subordination algebras}\label{sec: id comple}

For a boolean algebra $B$, let $\mathcal I(B)$ be the set of ideals of $B$ ordered by inclusion. It is well known that $\mathcal I(B)$ is a frame, where $I\wedge J=I\cap J$ and $\bigvee I_\alpha$ is the ideal generated by $\bigcup I_\alpha$. Moreover, the compact elements of $\mathcal I(B)$ are the principal ideals. This in particular implies that $\mathcal I(B)$ is compact and regular.\footnote{The frame $\mathcal I(B)$ is even zero-dimensional because every element in $\mathcal I(B)$ is a join of complemented elements (see \cite{Ban89}).} 
In this section we generalize these results to the frame of round ideals of an $\mathsf{S5}$-subordination algebra. 

Round ideals have been extensively studied in pointfree topology and domain theory. In particular, it follows from \cite{War75,Smy92} that the round ideals of a proximity lattice form a stably compact frame. As we pointed out in the previous section, $\mathsf{S5}$-subordination algebras $(B,S)$ are exactly the proximity algebras of \cite{GK81} where the algebra $B$ is a boolean algebra. This additional feature allows us to show that the round ideals of $(B,S)$ form a compact regular frame. Moreover, associating with each $\mathsf{S5}$-subordination algebra its frame of round ideals defines a contravariant functor from $\SubSfive$ to $\KRFrmR$. In \cref{sec:MacNeille} we will show that this functor is in fact a dual equivalence. 

\begin{definition}
	Let $\mathbf B=(B,S)$ be an $\mathsf{S5}$-subordination algebra. We call an ideal $I$ of $B$ a {\em round ideal} if $a\in I$ implies $a\S b$ for some $b\in I$. Let $\RI(\mathbf B)$ be the set of round ideals of $\mathbf{B}$ ordered by inclusion.
\end{definition}

\begin{remark}\label{rem:round ideal}
\hfill
\begin{enumerate}
\item\label{rem:round ideal:item1}
It is straightforward to see that an ideal $I$ is round iff $I=S^{-1}[I]$, and that if $I$ is an ideal of $B$, then $S^{-1}[I]$ is a round ideal of $\mathbf B$.
\item\label{rem:round ideal:item2} The notion of a round filter is dual to that of a round ideal. Therefore, a filter $F$ is round iff $F=S[F]$, and if $F$ is a filter of $B$, then $S[F]$ is a round filter of $\mathbf B$.
\end{enumerate}
\end{remark}

Let $B$ be a boolean algebra and $X \subseteq B$. We denote by $U(X)$ the set of upper bounds of $X$, by $L(X)$ the set of lower bounds of $X$, and by $\neg X$ the set $\{ \neg x \mid x \in X \}$. It is well known that $U(X)$ is a filter, $L(X)$ is an ideal, $\neg\neg X=X$, and 
$X$ is a filter iff $\neg X$ is an ideal. Moreover, $\neg U(X)=L(\neg X)$ and $\neg L(X)=U(\neg X)$.

\begin{lemma}\label{lem:operators and negations}
Let $B$ be a boolean algebra and $S$ an $\mathsf{S5}$-subordination on $B$. If $X \subseteq B$, then 
$\neg S[X]=S^{-1}[\neg X]$.
\end{lemma}

\begin{proof}
We have that $a \in \neg S[X]$ iff there is $x \in X$ such that $x \S \neg a$. By  \eqref{S6} this is equivalent to the existence of $x \in X$ such that $a \S \neg x$, which means that $a \in S^{-1}[\neg X]$.
\end{proof}

\begin{theorem}\label{thm:Round ideals frame}
Let $\mathbf B$ be an $\mathsf{S5}$-subordination algebra.
\begin{enumerate}[label=\normalfont(\arabic*), ref = \arabic*]
\item\label{thm:Round ideals frame:item1} $\RI(\mathbf B)$ is a subframe of $\mathcal I(\mathbf B)$.
\item\label{thm:Round ideals frame:item2} If $I\in\RI(\mathbf{B})$, then $I^*=S^{-1}[\neg U(I)] = \neg S[U(I)]$.
\item\label{thm:Round ideals frame:item3} The well-inside relation on $\RI(\mathbf B)$ is given by $I \prec J$ iff $U(I) \cap J \neq \varnothing$.
\item\label{thm:Round ideals frame:item4} $\RI(\mathbf{B})$ is compact and regular.
\end{enumerate}
\end{theorem} 

\begin{proof}
\eqref{thm:Round ideals frame:item1}. This follows from \cite[Thm.~3]{War75} (see also \cite[Thm.~1]{Smy92}).

\eqref{thm:Round ideals frame:item2}. The first equality follows from \cite[Thm.~3]{War75} and the second from \cref{lem:operators and negations}.

\eqref{thm:Round ideals frame:item3}.
By definition, $I \prec J$ iff $I^* \vee J=B$. By item~\eqref{thm:Round ideals frame:item2}, this is equivalent to $\neg S[U(I)] \vee J =B$, which  
holds iff there are $a \in S[U(I)]$ and $b \in J$ such that $\neg a \vee b=1$. Since $B$ is a boolean algebra, $\neg a \vee b = 1$ iff $a \le b$. 
Because $S[U(I)]$ is a filter (see \cref{rem:round ideal}\eqref{rem:round ideal:item2}), the existence of $a \in S[U(I)]$ with $a \le b$ is equivalent to $b \in S[U(I)]$. Thus, $I \prec J$ iff $S[U(I)] \cap J \neq \varnothing$.
We have that $S[U(I)] \cap J \neq \varnothing$ iff $U(I) \cap S^{-1}[J] \neq \varnothing$. Since $J$ is a round ideal, this is equivalent to $U(I) \cap J \neq \varnothing$.

\eqref{thm:Round ideals frame:item4}. That $\RI(\mathbf B)$ is compact follows from item~\eqref{thm:Round ideals frame:item1}. It follows from \cite[Thm.~3]{War75} that the relation on $\RI(\mathbf B)$ given by $U(I)\cap J\ne\varnothing$ is approximating. Thus, item~\eqref{thm:Round ideals frame:item3} implies that the well-inside relation is approximating, and hence $\RI(\mathbf B)$ is regular.
\end{proof}

Let $\mathbf{B}_1$ and $\mathbf{B}_2$ be $\mathsf{S5}$-subordination algebras and $T \colon \mathbf{B}_1 \to \mathbf{B}_2$ a compatible subordination. We define $\RI(T) \colon \RI(\mathbf{B}_2) \to \RI(\mathbf{B}_1)$ by setting $\RI(T)(I)=T^{-1}[I]$ for each round ideal $I$ of $\mathbf{B}_2$. 

\begin{theorem} \label{thm: RI contravariant}
$\RI\colon \SubSfive \to \KRFrmP$ is a well-defined contravariant functor.
\end{theorem}

\begin{proof}
That $\RI$ is well defined on objects follows from \cref{thm:Round ideals frame}\eqref{thm:Round ideals frame:item4}. We show that it is well defined on morphisms.
Let $T$ be a compatible subordination from $\mathbf{B}_1=(B_1,S_1)$ to $\mathbf{B}_2=(B_2,S_2)$. Let $I \in \RI(\mathbf{B}_2)$. Since $T$ is a subordination, it is straightforward to see that $T^{-1}[I]$ is an ideal. Because $T$ is compatible, $S_1^{-1} T^{-1}[I]= (T \circ S_1)^{-1}[I]=T^{-1}[I]$, and hence $T^{-1}[I]$ is a round ideal. Thus, $\RI(T)$ is well defined. 
To show that $\RI(T)$ is a preframe homomorphism, we need to prove that it preserves directed joins and finite meets. That it preserves directed joins is straightforward because directed joins are set-theoretic unions in $\I(\mathbf{B}_1)$ and $\I(\mathbf{B}_2)$, and hence also in their subframes $\RI(\mathbf{B}_1)$ and $\RI(\mathbf{B}_2)$. 
Moreover, we have that $T^{-1}[B_2]=B_1$ because $a \T 1$ for each $a \in B_1$. Thus, it remains to show that $\RI(T)$ preserves binary meets. Let $I,J \in \RI(\mathbf{B}_2)$. Clearly $T^{-1}[I \cap J] \subseteq T^{-1}[I] \cap T^{-1}[J]$. For the other inclusion, let $a \in T^{-1}[I] \cap T^{-1}[J]$. Then there are $b \in I$, $c \in J$ such that $a \T b$ and $a \T c$. Therefore, $a \T (b \wedge c) \in I \cap J$ by \eqref{S3}, and hence $a \in T^{-1}[I \cap J]$.

It is straightforward to show that $\RI$ preserves identities and reverses compositions. Thus, $\RI\colon \SubSfive \to \KRFrmP$ is a well-defined contravariant functor.
\end{proof}

In the next section we will show that $\RI$ is a dual equivalence.

\section{MacNeille completions of $\mathsf{S5}$-subordination algebras}\label{sec:MacNeille}

In \cite{ABC22a} we showed that the categories $\SubSfive$ and $\devS$ are equivalent. This was done by observing that each of these categories is equivalent to $\KHausR$. In this section we show that the equivalence can be obtained directly by generalizing the theory of MacNeille completions of boolean algebras to $\mathsf{S5}$-subordination algebras.

For a frame $L$, we recall (see, e.g., \cite{BP96}) that the {\em booleanization} of $L$ is 
\[
\bool L=\{ a\in L \mid a=a^{**} \},
\] 
and that $(\bool L,\sqcap,\bigsqcup)$ is a boolean frame (complete boolean algebra), where
\[
a\sqcap b=a\wedge b \quad \mbox{and} \quad \bigsqcup S = \left(\bigvee S\right)^{**}.
\]
If $L$ is compact regular, then $(\bool L,\prec)$ is a de Vries algebra, where $\prec$ is the restriction of the well-inside relation on $L$ to $\bool L$. As was shown in \cite{Bez12}, this correspondence extends to a covariant functor $\bool\colon\KRFrm\to\dev$ which is an equivalence. In the more general setting of $\KRFrmP$ and $\devS$, this correspondence extends to a contravariant functor as follows. 

Let $\Box \colon L \to M$ be a preframe homomorphism. Define the relation $\bool(\Box) \colon \bool M \to \bool L$ by 
\[
b \mathrel{\bool(\Box)} a \iff b \prec \Box a.
\]

\begin{lemma}\label{lem:boolf compatible}
If $\Box \colon L \to M$ is a preframe homomorphism, then $\bool(\Box) \colon \bool M \to \bool L$ is a compatible subordination relation.
\end{lemma}

\begin{proof}
	Let $T = \bool(\Box)$. It is straightforward to check that $T$ is a subordination. We only verify \eqref{S3}. Suppose $b \T a, c$. Then $b \prec \Box a$ and $b \prec \Box c$. Since $\Box$ is a preframe homomorphism, we have $b \prec \Box a \wedge \Box c = \Box(a \wedge c)$. Thus, $T$ satisfies \eqref{S3}. We next prove that $T$ is compatible. Let $a \in \bool L$ and $b \in \bool M$.
	We show that $b \T a$ iff there is $d \in \bool M$ such that $b \prec d \T a$. First suppose that $b \T a$, so $b \prec \Box a$. Since $M$ is compact regular, there is $d \in \bool M$ such that $b \prec d \prec \Box a$ (see, e.g., \cite[Rem.~3.2]{Bez12}). Therefore, $b \prec d \T a$. Conversely, suppose that $b \prec d \T a$. Then $b \prec d \prec \Box a$. Thus, $b \prec \Box a$, and so $b \T a$.

	It remains to show that $b \T a$ iff there is $c \in \bool L$ such that $b \T c \prec a$. For the right-to-left implication, we have that $c \prec a$ implies $c \le a$, and hence $\Box c \le \Box a$ because $\Box$ is order-preserving. Since $b \prec \Box c$, it follows that $b \prec \Box a$, and so $b \T a$. For the left-to-right implication, since $L$ is a regular frame, $a$ is the directed join of $\{c \in \bool L \mid c \prec a \}$. Therefore, since $\Box$ preserves directed joins, $\Box a = \bigvee \{ \Box c \mid c \in \bool L, \ c \prec a \}$. 
	Thus, from $b \prec \Box a$, using compactness, we find $c \in \bool L$ such that $c \prec a$ and $b \prec \Box c$.
\end{proof}

We thus define $\bool \colon \KRFrmP \to \devS$ by sending each compact regular frame $L$ to $(\bool L,\prec)$ and each preframe homomorphism $\Box\colon L\to M$ to $\bool(\Box)$.

\begin{proposition} \label{prop: B contravariant}
$\bool \colon \KRFrmP \to \devS$ is a contravariant functor.
\end{proposition}

\begin{proof}
That $\bool$ is well defined on objects follows from \cite[Lem.~3.1]{Bez12} and that it is well defined on morphisms from Lemma~\ref{lem:boolf compatible}. 
Let $L$ be a compact regular frame. If $\Box$ is the identity on $L$, then $\bool(\Box)$ coincides with $\prec$ which is the identity on $(\mathfrak B L,\prec)$. 
Let $\Box_1 \colon L \to M$ and $\Box_2 \colon M \to N$ be two preframe homomorphisms between compact regular frames. 
We show that $\bool(\Box_2 \circ \Box_1) = \bool(\Box_1) \circ \bool(\Box_2)$. Let $T_1=\bool(\Box_1)$ and $T_2=\bool(\Box_2)$.
For $a \in \bool L$ and $c \in \bool N$, if $c \mathrel{(T_1 \circ T_2)} a$, then there is $b \in \bool M$ such that $c \T_2 b$ and $b \T_1 a$. 
Thus, $c \prec \Box_2 b$ and $b \prec \Box_1 a$. 
Since $b \prec \Box_1 a$ and $\Box_2$ is order-preserving, we have $\Box_2 b \le \Box_2\Box_1 a$. 
Therefore, $c \prec \Box_2\Box_1 a$ which means that $c \mathrel{\bool(\Box_2 \circ \Box_1)} a$. 
Suppose next that $c \mathrel{\bool(\Box_2 \circ \Box_1)} a$. 
Therefore, $c \prec \Box_2 \Box_1 a$. By arguing as at the end of the proof of \cref{lem:boolf compatible}, there is $b \in \bool M$ such that $c \T_2 b$ and $b \prec \Box_1 a$. 
Thus, $c \T_2 b$ and $b \T_1 a$ which means that $c \mathrel{(T_1 \circ T_2)} a$.
\end{proof}

\begin{definition}
Let $\NI = \bool \circ \RI$.
\[
\begin{tikzcd}
\SubSfive \arrow[r, "\RI"'] \arrow[rr, bend left=25, "\NI"] &  \KRFrmP \arrow[r, "\bool"'] & \devS
\end{tikzcd}
\]
\end{definition}

By \cref{thm: RI contravariant} $\RI \colon \SubSfive \to \KRFrmP$ is a contravariant functor, and by \cref{prop: B contravariant} $\bool \colon \KRFrmP \to \devS$ is a contravariant functor. Thus, $\NI \colon \SubSfive \to \devS$ is a covariant functor. In particular, we have

\begin{proposition}\label{prop:B in subs5 then NI(B) in devs}
If $\mathbf{B}$ is an $\mathsf{S5}$-subordination algebra, then $\mathcal{NI}(\mathbf{B})$ is a de Vries algebra.
\end{proposition}

\begin{remark}\label{rem:prec on NIB}
Since $\prec$ on $\NI(\mathbf{B})$ is the restriction of $\prec$ on $\RI(\mathbf{B})$, by \cref{thm:Round ideals frame}\eqref{thm:Round ideals frame:item3} we have that $I \prec J$ iff $U(I) \cap J \neq \varnothing$ for all $I, J \in \NI(\mathbf{B})$.
\end{remark}

\begin{definition}\label{def:NI(B)}
Let $\mathbf{B}=(B,S)$ be an $\mathsf{S5}$-subordination algebra. We call $\mathcal{NI}(\mathbf{B})$ the \emph{MacNeille completion} of $\mathbf{B}$. We say that a round ideal $I$ of $\mathbf{B}$ is \emph{normal} if $I \in \NI(\mathbf{B})$. 
\end{definition}

The next theorem provides a characterization of normal round ideals.

\begin{theorem}\label{thm:NI booleanization of SI}
Let $I \in \RI(\mathbf{B})$. We have that $I \in \NI(\mathbf{B})$ iff $I=S^{-1}[L(S[U(I)])]$.
\end{theorem}

\begin{proof}
By \cref{lem:operators and negations} and \cref{thm:Round ideals frame}\eqref{thm:Round ideals frame:item2},
\begin{align*}
I^{**} &=\neg S[U(\neg S[U(I)])]=\neg S[\neg L(S[U(I)])]=\neg \neg S^{-1}[L(S[U(I)])]=S^{-1}[L(S[U(I)])].
\end{align*}
Since $I \in \NI(\mathbf{B})$ iff $I = I^{**}$, the result follows.
\end{proof}

\begin{remark}
We recall 
(see, e.g., \cite[p.~98]{Gra78})
that an ideal $I$ of a boolean algebra $B$ is {\em normal} if $LU(I)=I$, and that the {\em MacNeille completion} of $B$ is constructed as the complete boolean algebra of normal ideals of $B$. \cref{def:NI(B),thm:NI booleanization of SI} are an obvious generalization of this. Indeed, if $S$ is the partial ordering of $B$, then $I \in \NI(\mathbf{B})$ iff $I$ is a normal ideal of $B$.
For further connection, see \cref{prop:MacNeille and compingent}.
\end{remark}

An important feature of the MacNeille completion of an $\mathsf{S5}$-subordination algebra $\mathbf{B}$ is that it is isomorphic to $\mathbf{B}$ in $\SubSfive$ (which happens because morphisms in $\SubSfive$ 
are not structure-preserving bijections; see \cite[Rem.~3.15(4)]{ABC22a}). To see this, we need the following lemma.
We freely use the fact that if $I,J \in \RI(\mathbf{B})$, then
\begin{equation}\label{eq:I**precJ}
I \prec J \; \Longrightarrow \; I^{**} \prec J,
\end{equation}
which is a consequence of $I^{***}=I^{*}$.

\begin{lemma} \label{l:normal-in-between}
Let $a \in \mathbf{B}$ and $J \in \RI(\mathbf{B})$. Then $a \in J$ iff there is $I \in \NI(\mathbf{B})$ such that $a \in I \prec J$.
\end{lemma}

\begin{proof}
	For the right-to-left implication, if $a \in I \prec J$, then $a \in I \subseteq J$, and hence $a \in J$.
	For the left-to-right implication, 
	since $J$ is a round ideal, there is $b\in J$ such that $a \S b$.
	We have $a \in S^{-1}[b]$ and $b \in U(S^{-1}[b])$.
Thus, $S^{-1}[b] \prec J$ by \cref{thm:Round ideals frame}\eqref{thm:Round ideals frame:item3}.
Let $I =(S^{-1}[b])^{**}$. Then $I \in \mathcal{NI}(\mathbf{B})$ and $a \in S^{-1}[b] \subseteq I$. Moreover, by \eqref{eq:I**precJ}, $S^{-1}[b] \prec J$ implies $I \prec J$.
	Consequently, $a \in I \prec J$.
\end{proof}

Let $Q_{\mathbf{B}} \colon \mathbf{B} \to \mathcal{NI}(\mathbf{B})$ be the relation defined by
\[
a \, Q_{\mathbf{B}} \, I \iff a \in I.
\] 

\begin{lemma}
	$Q_{\mathbf{B}}$ is a morphism in $\SubSfive$.
\end{lemma}

\begin{proof}
	It is easy to see that $Q_{\mathbf{B}}$ is a subordination relation.
	The equality $Q_{\mathbf{B}} = Q_{\mathbf{B}} \circ S$ follows from $I = S^{-1}[I]$, and the equality ${\prec} \circ Q_{\mathbf{B}} = Q_{\mathbf{B}}$ from \cref{l:normal-in-between}.
\end{proof}

If $T\colon\mathbf{B}_1 \to \mathbf{B}_2$ is a morphism in $\SubSfive$, define $\inv{T}\colon\mathbf{B}_2\to\mathbf{B}_1$ by 
\begin{equation}\label{eq:T inv}
b\mathrel{\inv{T}}a \iff \neg a \T \neg b.
\end{equation}
Then $\inv{T}$ is a morphism in $\SubSfive$ (see the paragraph before \cite[Thm.~3.10]{ABC22a}).

\begin{lemma}\label{thm:iso}
	$Q_{\mathbf{B}} \colon \mathbf{B} \to \mathcal{NI}(\mathbf{B})$ is an isomorphism.
\end{lemma}

\begin{proof}
	Let $T = \inv{Q_{\mathbf{B}}} \colon \mathcal{NI}(\mathbf{B}) \to \mathbf{B}$. By \eqref{eq:T inv} and \cref{thm:Round ideals frame}\eqref{thm:Round ideals frame:item2},
	\begin{equation}\label{QA}
	I \mathrel{T} a \iff \neg a \mathrel{Q_{\mathbf{B}}} I^* \iff \neg a \in \neg S[U(I)] \iff a \in S[U(I)].	
	\end{equation}
	We show that $Q_{\mathbf{B}}$ and $T$ are inverses of each other.
	For this we need to prove that $T \circ Q_{\mathbf{B}} = S$ and $Q_{\mathbf{B}} \circ T = \, \prec$.	
	
	We first show that $T \circ Q_{\mathbf{B}} = S$. 
	For the inclusion $\subseteq$, 
	let $a,b \in B$, $I \in \mathcal{NI}(\mathbf{B})$, and $a \Q_{\mathbf{B}} I \T  b$. Then $a \in I$ and $b \in S[U(I)]$ by \eqref{QA}.
	Thus, $a \S b$.	
	For the inclusion $\supseteq$, let $a,b \in B$ with $a \S b$.  Then $a \in S^{-1}[b]$ and \cref{l:normal-in-between} implies that there is $I \in \NI(\mathbf{B})$ such that $a \in I \prec S^{-1}[b]$.
  By \cref{rem:prec on NIB} and \eqref{QA}, 
  \begin{equation*}
  I\prec S^{-1}[b] \iff U(I) \cap S^{-1} [b]\neq \varnothing \iff  b \in S[U(I)] \iff I \T b.    
  \end{equation*}
  Thus, $a \Q_{\mathbf{B}} I \mathrel{T} b$.
		
	We next show that $Q_{\mathbf{B}} \circ T = \, \prec$.
	Let $I, J \in \mathcal{NI}(\mathbf{B})$.
	By \cref{rem:prec on NIB} and \eqref{QA},
	\begin{align*}
	I \prec J & \iff U(I) \cap J \neq \varnothing \iff U(I) \cap S^{-1}[J] \neq \varnothing \iff S[U(I)] \cap J \neq \varnothing\ \\
	& \iff \exists a \in S[U(I)] \cap J \iff \exists a \in B : I \T a \Q_{\mathbf{B}} J \iff I \mathrel{(Q_{\mathbf{B}} \circ T)} J. 
	\end{align*}
	Thus, $Q_{\mathbf{B}} \colon \mathbf{B} \to \mathcal{NI}(\mathbf{B})$ is an isomorphism.
\end{proof}

\begin{proposition}\label{thm:Q nat iso}
Let $\F\colon  \devS \to \SubSfive$ be the inclusion functor. Then $Q \colon 1_{\SubSfive} \to \F \circ\NI$ is a natural isomorphism.
\end{proposition}

\begin{proof}
Let 
$T \colon \mathbf{B}_1 \to \mathbf{B}_2$ 
be a morphism in $\SubSfive$. By \cref{thm:iso}, it is sufficient to show that $\NI(T) \circ Q_{\mathbf{B}_1} = Q_{\mathbf{B}_2} \circ T$. (Since $\F$ is the inclusion functor, we omit it from the diagram.)
\[
\begin{tikzcd}
\mathbf{B}_1 \arrow[d, "T"'] \arrow[r, "Q_{\mathbf{B}_1}"] &  \NI(\mathbf{B}_1) \arrow[d, "\NI(T)"] \\
\mathbf{B}_2 \arrow[r, "Q_{\mathbf{B}_2}"] & \NI(\mathbf{B}_2)
\end{tikzcd}
\]
Let $a \in B_1$ and $I \in \NI(\mathbf{B}_2)$. 
We have 
\[
a \mathrel{(\NI(T) \circ Q_{\mathbf{B}_1})} I \iff \exists J \in \NI(\mathbf{B}_1) : a \in J \mbox{ and } J \prec T^{-1}[I],
\]
and
\[
a \mathrel{(Q_{\mathbf{B}_2} \circ T)} I \iff \exists b \in B_2 : a \T b \mbox{ and } b \in I  \iff a \in T^{-1}[I].
\] 
The two conditions are equivalent by \cref{l:normal-in-between}.
\end{proof}

\begin{theorem} \label{thm: NI equivalence}
The functors $\NI \colon \SubSfive \to \devS$ and $\F \colon \devS \to \SubSfive$ are quasi-inverses of each other.
Thus, $\SubSfive$ and $\devS$ are equivalent.
\end{theorem}

\begin{proof}
By \cref{thm:Q nat iso}, $Q \colon 1_{\SubSfive} \to \F \circ\NI$ is a natural isomorphism. 
For the same reason, we have a natural isomorphism $Q' \colon 1_{\devS} \to \NI \circ \F$ whose component on $\mathbf{B} \in \devS$ is $Q_\mathbf{B}$.
Thus, $\F\colon  \devS \to \SubSfive$ is a quasi-inverse of $\NI$. 
\end{proof}

\cref{thm: NI equivalence}
gives a direct choice-free proof that $\SubSfive$ is equivalent to $\devS$. We next show that when restricted to compingent algebras, $\NI$ yields the usual MacNeille completion.

\begin{proposition}\label{prop:MacNeille and compingent}
Let $\mathbf{B} = (B,S)$ be an $\mathsf{S5}$-subordination algebra.
\begin{enumerate}[label=\normalfont(\arabic*), ref = \arabic*]
\item\label{prop:MacNeille and compingent:item1} If $\mathbf{B}$ is a compingent algebra, then there is a boolean isomorphism between $\NI(\mathbf{B})$ and the usual MacNeille completion $\overline{B}$ of $B$.
\item\label{prop:MacNeille and compingent:item2} If $\mathbf{B}$ is a de Vries algebra, then there is a structure-preserving bijection between $\mathbf{B}$ and $\NI(\mathbf{B})$.
\end{enumerate}
\end{proposition}

\begin{proof}
\eqref{prop:MacNeille and compingent:item1}. Since $\mathbf{B}$ is a compingent algebra, it follows from~\cite[Thm.~1.1.4]{deV62} that each $b \in B$ is the supremum of $S^{-1}[b]$. We use this fact to prove that
\begin{equation}\label{eq:US-1I=UI}
U(S^{-1}[I]) = U(I)
\end{equation}
for each ideal $I$ of $B$.
Since $S^{-1}[I] \subseteq I$, we have $U(I) \subseteq U(S^{-1}[I])$. For the reverse inclusion, let $a \in U(S^{-1}[I])$. We show that $a \in U(I)$. Let $b \in I$. Then $S^{-1}[b] \subseteq S^{-1}[I]$. Therefore, $a \in U(S^{-1}[b])$, so $a \ge \bigvee S^{-1}[b]=b$. Thus, $a \in U(I)$. This proves \eqref{eq:US-1I=UI}. A similar argument proves that 
\begin{equation}\label{eq:LSF=LF}
L(S[F]) = L(F)
\end{equation}
for each filter $F$ of $B$.
	By \eqref{eq:US-1I=UI} and \eqref{eq:LSF=LF}, for every normal 
	ideal $I$ of $B$, we have 
	\[
	L(S[U(S^{-1}[I])]) = L(S[U(I)]) = L(U(I)) = I.
	\]
	Thus, applying $S^{-1}$ to both sides yields 
	\[
	S^{-1}[L(S[U(S^{-1}[I])])] = S^{-1}[I].
	\]
	This shows, by \cref{thm:NI booleanization of SI}, that $S^{-1}[I] \in \NI(\mathbf{B})$ for every normal 
	ideal $I$ of $B$.
	This defines an order-preserving map $\alpha \colon \overline{B} \to \NI(\mathbf{B})$.
	 
	Conversely, for every $I \in \NI(\mathbf{B})$, we have that
	$L(U(I))$ is a normal 
	ideal of $B$. 
	This defines an order-preserving map $\beta \colon \NI(\mathbf{B}) \to \overline{B}$.
	By \eqref{eq:US-1I=UI}, for a normal ideal $I$ of $B$, we have
	\[
	L(U(S^{-1}[I])) = L(U(I)) = I.
	\] 
	For a normal round ideal $I$, by \eqref{eq:LSF=LF} and \cref{thm:NI booleanization of SI}, we have
	\[
	S^{-1}[L(U(I))] = S^{-1}[L(S[U(I)]) = I.
	\] 
Thus, $\alpha$ and $\beta$ are order-isomorphisms, hence boolean isomorphisms.

	\eqref{prop:MacNeille and compingent:item2}.
	It is well known (see, e.g., \cite[Thm.~22]{GH09}) that sending $b$ to the downset 
	${\downarrow} b \coloneqq \{ a \in B \mid a \le b \}$
	gives a boolean embedding of $B$ into $\overline{B}$, which is an isomorphism iff $B$ is complete. Composing with $\alpha$ yields the boolean embedding $\iota \colon B \to \NI(\mathbf{B})$ given by $\iota(b)=S^{-1}[b]$. If $\mathbf{B}$ is a de Vries algebra, then $\iota$ becomes a boolean isomorphism by item~\eqref{prop:MacNeille and compingent:item1}.
	It is left to prove that $a \S b$ iff $\iota(a) \prec \iota(b)$. If $a \S b$, then $a \in U(\iota(a)) \cap \iota(b)$, and so $\iota(a) \prec \iota(b)$ by \cref{rem:prec on NIB}. Conversely, suppose that $\iota(a) \prec \iota(b)$. Then $U(\iota(a)) \cap \iota(b) \neq \varnothing$, so there exists $c \in U(\iota(a)) \cap \iota(b)$. Since $a$ is the supremum of $\iota(a) = S^{-1}[a]$, we have that $a \le c \S b$, and hence $a \S b$.
	Thus, $\iota$ is a structure-preserving bijection between $\mathbf{B}$ and $\NI(\mathbf{B})$.
\end{proof}

\begin{remark}
Let $\mathbf{B}=(B,S)$ be a compingent algebra and $\overline{B}$ the MacNeille completion of $B$. By \cite[Rem.~5.11]{BBSV19}, $(\overline{B},\lhd)$ is a de Vries algebra, where
\[
I \lhd J \iff U(I)\cap S^{-1}[J] \neq \varnothing.
\] 
A straightforward verification shows that the boolean isomorphism of \cref{prop:MacNeille and compingent}\eqref{prop:MacNeille and compingent:item1} is an isomorphism of de Vries algebras between $\NI(\mathbf{B})$ and $(\overline{B},\lhd)$.
\end{remark}

\begin{remark}
Let $\mathbf{B}$ be a compingent algebra. Then $Q_{\mathbf{B}} \colon \mathbf{B} \to \NI(\mathbf{B})$ and $\iota \colon \mathbf{B} \to \NI(\mathbf{B})$ are related as follows:
\[
a \Q_{\mathbf{B}} I \iff \iota(a) \prec I
\]
for each $a \in B$ and $I \in \NI(\mathbf{B})$. Indeed, since $\mathbf{B}$ is a compingent algebra, $a = \bigvee S^{-1}[a]$, so ${\uparrow} a = U(S^{-1}[a])$, and hence
\[
a \Q_{\mathbf{B}} I \iff a \in I \iff {\uparrow} a \cap I \neq \varnothing \iff U(S^{-1}[a]) \cap I \neq \varnothing \iff \iota(a) \prec I.
\]
\end{remark}

We finish the section by proving that both $\SubSfive$ and $\devS$ are dually equivalent to $\KRFrmP$.
Let $L \in \KRFrmP$. By \cite[Rem.~3.10]{Bez12}, the map $f_L \colon L \to \RI(\bool L)$ given by 
\[
f_L(a) = \{ b \in \bool L \mid b \prec a \}
\]
is an isomorphism of frames.

\begin{proposition}\label{prop:f nat iso}
$f\colon  1_{\KRFrmP} \to \RI \circ \F \circ \bool$ is a natural isomorphism.
\end{proposition}

\begin{proof}
Let $\Box \colon L \to M$ be a preframe homomorphism and $T=\bool(\Box)$. Because each $f_L$ is an isomorphism, it is enough to show that $\RI(T) \circ f_L = f_M \circ \Box$. (Since $\F$ is the inclusion functor, we omit it from the diagram.)
\[
\begin{tikzcd}
L \arrow[d, "\Box"'] \arrow[r, "f_L"] &  \RI(\bool L) \arrow[d, "\RI(T)"] \\
M \arrow[r, "f_M"] & \RI(\bool M)
\end{tikzcd}
\]
Let $a \in L$. 
We have
\begin{align*}
\RI(T)(f_L(a)) = T^{-1}[f_L(a)] & = \{ b \in \bool M \mid \exists c \in \bool L : b \T c, \ c \prec a \}\\
& = \{ b \in \bool M \mid \exists c \in \bool L : b \prec \Box c, \ c \prec a  \},
\end{align*}
and $f_M(\Box a) = \{ b \in \bool M \mid b \prec \Box a \}$. An argument similar to the last paragraph of the proof of \cref{lem:boolf compatible} yields
\[
\{ b \in \bool M \mid \exists c \in \bool L : b \prec \Box c, \ c \prec a \} = \{ b \in \bool M \mid b \prec \Box a \},
\]
completing the proof.
\end{proof}

\begin{theorem} \label{thm:dual-equivalence}
\mbox{}\begin{enumerate}[label=\normalfont(\arabic*), ref = \arabic*]
\item\label{thm:dual-equivalence:item1} $\RI$ and $\F \circ \bool$ form a dual equivalence between $\SubSfive$ and $\KRFrmP$.
\item\label{thm:dual-equivalence:item2}  $\RI \circ \F$ and $\mathfrak{B}$ form a dual equivalence between $\devS$ and $\KRFrmP$.
\end{enumerate}
We thus obtain the following diagram of equivalences and dual equivalences that commutes up to natural isomorphism.
\[
\begin{tikzcd}[column sep=6ex, row sep=6ex]
& \SubSfive \arrow[dl, "\RI"'] \arrow[dr, shift left=1ex, "\NI"] & \\
\KRFrmP \arrow[rr, "\bool"] & & \devS \arrow[ul, hookrightarrow, shift left=0.5ex, "\F"]
\end{tikzcd}
\]
\end{theorem}

\begin{proof}
\eqref{thm:dual-equivalence:item1}. By definition of $\NI$, we have $\F \circ \bool \circ \RI = \F \circ \NI$. Therefore, $Q \colon 1_{\SubSfive} \to \F \circ \bool \circ \RI$ is a natural isomorphism by \cref{thm:Q nat iso}. Moreover, $f\colon  1_{\KRFrmP} \to \RI \circ \F \circ \bool$ is a natural isomorphism by \cref{prop:f nat iso}. Thus, $\F \circ \bool \colon \KRFrmP \to \SubSfive$ is a quasi-inverse of $\RI$. 

\eqref{thm:dual-equivalence:item2}. By \cref{thm:Q nat iso}, $Q \colon 1_{\SubSfive} \to \F \circ \bool \circ \RI$ is a natural isomorphism. 
For the same reason, we have a natural isomorphism $Q' \colon 1_{\devS} \to \bool \circ \RI \circ \F$ whose component on $\mathbf{B} \in \devS$ is $Q_\mathbf{B}$.
Thus, $\bool \colon \KRFrmP \to \devS$ is a quasi-inverse of $\RI \circ \F$.
\end{proof}

\section{Continuous subordinations}\label{sec: cont subs}

In \cref{sec:MacNeille} we gave a direct choice-free proof that $\SubSfive$ is equivalent to $\devS$ and dually equivalent to $\KRFrmP$. Morphisms of each of these categories encode closed relations between compact Hausdorff spaces. In this section we study the wide subcategories of these categories whose morphisms encode continuous relations between compact Hausdorff spaces.

Recalling from \cref{rem:functors equivalences} that we have an equivalence $\mathcal{Q}\colon \StoneER \to \KHausR$, we first characterize when $\mathcal{Q}(R)$ is a continuous relation for an arbitrary morphism $R$ in $\StoneER$.
We then use the equivalence $\Clop \colon \StoneER \to \SubSfive$ to encode this characterization in the language of $\mathsf{S5}$-subordination algebras.

\begin{definition}
Let $R$ be a binary relation on a set $X$ and $U \subseteq X$. Following the standard notation in modal logic, we write $\Box_R U = X\setminus R^{-1}[X\setminus U]$.
If $R$ is an equivalence relation, we say that $U$ is \emph{$R$-saturated} if $R[U]=U$.
\end{definition}

\begin{remark}\label{rem:box}
\mbox{}\begin{enumerate}
\item\label{rem:box:item1} If $R$ is a closed relation and $U$ is open, then $\Box_R U$ is open.
\item\label{rem:box:item2} If $R$ is an equivalence relation, then $\Box_R U = X\setminus R[X\setminus U]$ and is the largest $R$-saturated subset of $U$. Therefore, $U$ is $R$-saturated iff $\Box_R U=U$. 
\end{enumerate}
\end{remark}

\begin{lemma} \label{t:correspondence continuous rel and sub}
	Let $R \colon (X_1, E_1) \to (X_2, E_2)$ be a morphism in $\StoneER$.
	The following are equivalent.
	\begin{enumerate}[label=\normalfont(\arabic*), ref = \arabic*]
		\item \label{i:continuous-relation}
		The relation $\mathcal{Q}(R) \colon X_1/E_1 \to X_2/E_2$ is a continuous relation.
		\item \label{i:saturated}
		If $V$ is an $E_2$-saturated open in $X_2$, then $R^{-1}[V]$ is open in $X_1$.
		\item \label{i:clopens}
		If $B_1,B_2\subseteq X_2$ are clopen with $E_2[B_1] \subseteq B_2$, then there is a clopen set $A\subseteq X_1$ such that $R^{-1}[B_1] \subseteq A \subseteq R^{-1}[B_2]$.
		\item \label{i:continuous-algebraic}
		If $B_1,B_2\subseteq X_2$ are clopen with $E_2[B_1] \subseteq B_2$, then there is a clopen set $A\subseteq X_1$ such that $A \in \inv{S_R}[B_1]$ and $\inv{S_R}[B_2] \subseteq S_{E_1}[A]$.
	\end{enumerate}
\end{lemma}

\begin{proof}
	
	\eqref{i:continuous-relation}$\Leftrightarrow$\eqref{i:saturated}.
	Let $\pi_i \colon X_i \to X_i/E_i$ be the quotient maps for $i=1,2$.
		\[
		\begin{tikzcd}
			X_1 \arrow{r}{R} \arrow[swap]{d}{\pi_1} & X_2 \arrow{d}{\pi_2}\\
			X_1/E_1 \arrow[swap]{r}{\mathcal{Q}(R)} & X_2/E_2
		\end{tikzcd}
		\]	
	 Then $\mathcal{Q}(R)^{-1}[U] = \pi_1[R^{-1}[\pi_2^{-1}[U]]]$ for each $U \subseteq X_2/E_2$. 
The $R$-inverse image of any subset of $X_2$ is $E_1$-saturated by the compatibility of $R$.
Thus, $R^{-1}[\pi_2^{-1}[U]]$ is open iff $\pi_1[R^{-1}[\pi_2^{-1}[U]]]$ is open for each $U$ open of $X_2/E_2$.
Therefore, $\mathcal{Q}(R)$ is continuous iff $R^{-1}[\pi_2^{-1}[U]]$ is open for each $U$ open of $X_2/E_2$. Since $V$ is an $E_2$-saturated open in $X_2$ iff $V=\pi_2^{-1}[U]$ for some $U$ open of $X_2/E_2$, the equivalence follows.

	\eqref{i:saturated}$\Rightarrow$\eqref{i:clopens}.
	Suppose $B_1,B_2\subseteq X_2$ are clopens with $E_2[B_1] \subseteq B_2$.
	Let $V = \Box_{E_2} B_2$.
	Then $V$ is an $E_2$-saturated open.
	Since $E_2[B_1] \subseteq B_2$, we have that $B_1 \subseteq V$.
	Therefore, $R^{-1}[B_1] \subseteq R^{-1}[V]$.
	The set $R^{-1}[B_1]$ is closed and $R^{-1}[V]$ is open by item~\eqref{i:saturated}.
	Thus, there is a clopen set	$A\subseteq X_1$ such that $R^{-1}[B_1] \subseteq A \subseteq R^{-1}[V]$.
	Since $V \subseteq B_2$, we have $R^{-1}[V] \subseteq R^{-1}[B_2]$.
	Hence, $A \subseteq R^{-1}[B_2]$. 
	This proves item~\eqref{i:clopens}.
	
	\eqref{i:clopens}$\Rightarrow$\eqref{i:saturated}.
	Suppose that $V$ is an $E_2$-saturated open subset of $X_2$.
	Since $V = \bigcup\{ B \in \Clop(X_2) \mid B \subseteq V \}$, we have
	\[
	R^{-1}[V] = \bigcup \{R^{-1}[B] \mid B \in \Clop(X_2), B \subseteq V\}.
	\]
	Thus, it is enough to prove that for every clopen subset $B$ of $X_2$ contained in $V$, there is an open subset $U_B$ of $X_1$ such that $R^{-1}[B] \subseteq U_B \subseteq R^{-1}[V]$ (because then $R^{-1}[V] = \bigcup \{U_B \mid B \in \Clop(X_2), B \subseteq V\}$).
	Let $B$ be a clopen subset of $X_2$ contained in $V$.
	Since $V$ is $E_2$-saturated, $E_2[B] \subseteq V$.
	Because $E_2[B]$ is closed and $V$ is open, there is a clopen subset $B'$ of $X_2$ such that $E_2[B] \subseteq B' \subseteq V$.
	By item~\eqref{i:clopens},	there is a clopen set $A\subseteq X_1$ such that $R^{-1}[B] \subseteq A \subseteq R^{-1}[B']$.
	Since $B' \subseteq V$, we have $R^{-1}[B'] \subseteq R^{-1}[V]$, so
	$A \subseteq R^{-1}[V]$.
	Therefore, we have found an open subset $A$ of $X_1$ such that $R^{-1}[B] \subseteq A \subseteq R^{-1}[V]$.
	Hence, item~\eqref{i:saturated} holds.
	
	\eqref{i:clopens}$\Leftrightarrow$\eqref{i:continuous-algebraic}. This follows from the following two claims.
	
	\begin{claim}\label{claim:6.2}
		For clopen sets $A\subseteq X_1$ and $B\subseteq X_2$, we have $R^{-1}[B] \subseteq A$ iff $A \in \inv{S_R}[B]$.
	\end{claim}
	
	\begin{proof}[Proof of claim]
		This follows from the equality $\inv{S_R} = S_{\conv{R}}$, shown in the proof of \cite[Thm.~2.14]{ABC22a}.
	\end{proof}
	
	\begin{claim}
		For clopen sets $A\subseteq X_1$ and $B\subseteq X_2$, we have $A \subseteq R^{-1}[B]$ iff $\inv{S_R}[B] \subseteq S_{E_1}[A]$.
	\end{claim}
	
	\begin{proof}[Proof of claim]
		Let $A\subseteq X_1$ and $B\subseteq X_2$ be clopen sets. Then
		\begin{align*}
			&\inv{S_R}[B] \subseteq S_{E_1}[A]\\
			& \iff \forall A' \in \Clop(X_1), \ B \mathrel{\inv{\S_R}} A' \text{ implies } A \S_{E_1} A'\\
			& \iff \forall A' \in \Clop(X_1), \ R^{-1}[B] \subseteq A' \text{ implies } E_1[A] \subseteq A' & \text{(by \cref{claim:6.2})}\\
			& \iff E_1[A] \subseteq \bigcap \{ A' \in \Clop(X_1) \mid R^{-1}[B] \subseteq A' \}\\
			& \iff E_1[A] \subseteq R^{-1}[B] & \text{(since $R^{-1}[B]$ is closed)}\\
			& \iff A \subseteq R^{-1}[B] & \text{(since $R^{-1}[B]$ is $E_1$-saturated).}
		\end{align*}
	\end{proof}
This concludes the proof.
\end{proof}

The next definition encodes \cref{t:correspondence continuous rel and sub}\eqref{i:continuous-algebraic} in the language of $\mathsf{S5}$-subordination algebras. By \cref{t:correspondence continuous rel and sub}\eqref{i:continuous-relation}, this condition is equivalent to the corresponding relation between compact Hausdorff spaces being continuous. Because of this, we call such compatible subordinations continuous.

\begin{definition}
Let $T \colon (B_1, S_1) \to (B_2, S_2)$ be a compatible subordination between $\mathsf{S5}$-subordination algebras.
We say that $T$ is \emph{continuous} if the following holds:
\[
\forall b_1, b_2 \in B_2 \text{ with } b_1 \S_2 b_2 \text{ there is } a \in \inv{T}[b_1] \text{ such that } \inv{T}[b_2] \subseteq S_1[a].
\]
\end{definition}

\begin{lemma}\label{rem:cont}
Let $T \colon (B_1, S_1) \to (B_2, S_2)$ be a compatible subordination.
\begin{enumerate}[label=\normalfont(\arabic*), ref = \arabic*]
\item\label{rem:cont:item1} The following are equivalent:
\begin{enumerate}[label=\normalfont(\alph*), ref = 1\alph*]
\item\label{rem:cont:item1:Tcontinuous} $T$ is continuous.
\item\label{rem:cont:item1:cond1} $\forall b_1, b_2 \in B_2 \text{ with } b_1 \S_2 b_2 \text{ there is } a \in \inv{T}[b_1] \text{ such that } a \in L(\inv{T}[b_2])$.
\item\label{rem:cont:item1:cond2} $\forall b_1, b_2 \in B_2 \text{ with } b_1 \S_2 b_2 \text{ there is } a \in T^{-1}[b_2] \text{ such that } a \in U(T^{-1}[b_1])$.
\end{enumerate}
\item\label{rem:cont:item2} If $B_1$ is complete, then the following are equivalent:
\begin{enumerate}[label=\normalfont(\alph*), ref = 2\alph*]
\item\label{rem:cont:item2:Tcontinuous} $T$ is continuous.
\item\label{rem:cont:item2:cond1} $\forall b_1, b_2 \in B_2 \text{ with } b_1 \S_2 b_2 \text{ we have } b_1 \mathrel{\inv{T}} (\bigwedge \inv{T}[b_2])$.
\item\label{rem:cont:item2:cond2} $\forall b_1, b_2 \in B_2 \text{ with } b_1 \S_2 b_2 \text{ we have } (\bigvee\T^{-1}[b_1]) \T b_2$.
\end{enumerate}
\end{enumerate}
\end{lemma}

\begin{proof}
\eqref{rem:cont:item1:Tcontinuous}$\Leftrightarrow$\eqref{rem:cont:item1:cond1}. It is enough to prove that $\inv{T}[b_2] \subseteq S_1[a]$ is equivalent to $a \in L(\inv{T}[b_2])$. For the left-to-right implication, 
by \eqref{S5} we have $S_1[a] \subseteq U(a)$, and so $\inv{T}[b_2] \subseteq S_1[a]$ implies $\inv{T}[b_2] \subseteq U(a)$, which is equivalent to $a \in L(\inv{T}[b_2])$.
For the right-to-left implication, suppose $a \in L(\inv{T}[b_2])$ and let $a' \in \inv{T}[b_2]$.
Since $\inv{T}$ is a compatible subordination, there is $a'' \in \inv{T}[b_2]$ such that $a'' \S_1 a'$.
Therefore, $a \leq a'' \S_1 a'$, which implies $a \S_1 a'$, and hence $a' \in S_1[a]$.

\eqref{rem:cont:item1:cond1}$\Leftrightarrow$\eqref{rem:cont:item1:cond2}.  Suppose that \eqref{rem:cont:item1:cond1} holds, and let $b_1, b_2 \in B_2$ be such that $b_1 \S_2 b_2$.
Then, by \eqref{S6}, $\lnot b_2 \S_2 \lnot b_1$.
Therefore, by \eqref{rem:cont:item1:cond1} there is $a \in \inv{T}[\lnot b_2]$ such that $a \in L(\inv{T}[\lnot b_1])$.
The condition $a \in \inv{T}[\lnot b_2]$ is equivalent to $\lnot a \in T^{-1}[b_2]$.
Similarly, 
the condition $a \in L(\inv{T}[\lnot b_1])$ is equivalent to $\lnot a \in U(T^{-1}[b_1])$.
Thus, \eqref{rem:cont:item1:cond1} implies \eqref{rem:cont:item1:cond2}, and the converse is proved similarly.

\eqref{rem:cont:item2}. If $B$ is complete, then \eqref{rem:cont:item1:cond1}$\Leftrightarrow$\eqref{rem:cont:item2:cond1} and \eqref{rem:cont:item1:cond2}$\Leftrightarrow$\eqref{rem:cont:item2:cond2}. Thus, the result follows from item~\eqref{rem:cont:item1}.
\end{proof}

\begin{lemma}\label{lem:continuous sub id and comp}
\mbox{}\begin{enumerate}[label=\normalfont(\arabic*), ref = \arabic*]
\item\label{lem:continuous sub id and comp:item1} Let $(B,S)$ be an $\mathsf{S5}$-subordination algebra. The identity morphism $S \colon (B,S) \to (B,S)$ in $\SubSfive$ is continuous.
\item\label{lem:continuous sub id and comp:item2} 	
Let $T_1 \colon (B_1, S_1) \to (B_2, S_2)$ and $T_2 \colon (B_2, S_2) \to (B_3, S_3)$ be continuous compatible subordinations between $\mathsf{S5}$-subordination algebras.
Then $T_2 \circ T_1 \colon (B_1, S_1) \to (B_3, S_3)$ is a continuous compatible subordination.
\end{enumerate}
\end{lemma}

\begin{proof}
\eqref{lem:continuous sub id and comp:item1}. Since $\inv{S}=S$, this is immediate from \eqref{S7}.

\eqref{lem:continuous sub id and comp:item2}.
	It is sufficient to show that $T_2 \circ T_1$ is continuous.
	Let $c_1,c_2 \in B_3$ be such that $c_1 \S_3 c_2$.
	By \eqref{S7}, there is $c \in B_3$ such that $c_1 \S_3 c \S_3 c_2$.
	Therefore, since $T_2$ is continuous, there are $b_1 \in \inv{T}_2[c_1]$ and $b_2 \in \inv{T}_2[c]$ such that $\inv{T}_2[c] \subseteq S_2[b_1]$ and $\inv{T}_2[c_2] \subseteq S_2[b_2]$.
	We have $b_2 \in \inv{T}_2[c] \subseteq S_2[b_1]$, and so $b_1 \S_2 b_2$.
	Thus, since $T_1$ is continuous, there is $a \in \inv{T}_1[b_1]$ such that $\inv{T}_1[b_2] \subseteq S_1[a]$.
	We have $c_1 \mathrel{\inv{T}_2} b_1 \mathrel{\inv{T}_1} a$, and hence $a \in(\inv{T}_1 \circ \inv{T}_2)[c_1]$.
	Since $\inv{T}_1 \circ \inv{T}_2=\inv{T_2 \circ T_1}$, it remains to show that $(\inv{T}_1 \circ \inv{T}_2)[c_2] \subseteq S_1[a]$.
	Let $a' \in (\inv{T}_1 \circ \inv{T}_2)[c_2]$. 
	Then there is $b \in B_2$ such that $c_2 \mathrel{\inv{T}_2} b \mathrel{\inv{T}_1} a'$.
	We have $b \in \inv{T}_2[c_2] \subseteq S_2[b_2]$, and thus $b_2 \S_2 b$.
	From $b_2 \S_2 b \mathrel{\inv{T}_1} a'$ we deduce, using the compatibility of $\inv{T}_1$, that $b_2 \mathrel{\inv{T}_1} a'$.
	Therefore, $a' \in \inv{T}_1[b_2] \subseteq S_1[a]$, 
	and hence $a' \in S_1[a]$, as desired.
\end{proof}

\begin{definition} \label{def:subsfiveC}
Let $\SubSfiveCS$ be the wide subcategory of $\SubSfive$ whose morphisms are continuous compatible subordinations, and define $\devCS$ similarly.
\end{definition}

We next show that \cref{thm:dual-equivalence} restricts to yield the corresponding dual equivalences for $\SubSfiveCS$ and $\devCS$. For this we need the following lemma.

\begin{lemma}\label{lem:continuity SI(T) and B(box)}
Let $(B_1,S_1),(B_2,S_2)$ be $\mathsf{S5}$-subordination algebras and $T \colon B_1 \to B_2$ be a morphism in $\SubSfive$. Let also $L_1,L_2$ be compact regular frames and $\Box \colon L_1 \to L_2$ a preframe homomorphism. 
\begin{enumerate}[label=\normalfont(\arabic*), ref = \arabic*]
\item\label{lem:continuity SI(T) and B(box):item1} If $T \colon B_1 \to B_2$ is continuous, then $\RI(T) \colon \RI(B_2,S_2) \to \RI(B_1,S_1)$ is a $\mathsf{c}$-morphism. 
\item\label{lem:continuity SI(T) and B(box):item2} If $\Box \colon L_1 \to L_2$ is a $\mathsf{c}$-morphism, then $\bool(\Box) \colon \bool(L_2) \to \bool(L_1)$ is  continuous.
\item\label{lem:continuity SI(T) and B(box):item3} If $T \colon B_1\to B_2$ is an isomorphism in $\SubSfive$, then $T$ is an isomorphism in $\SubSfiveCS$.
\item\label{lem:continuity SI(T) and B(box):item4} If $\Box \colon L_1\to L_2$ is an isomorphism in $\KRFrmP$, then $\Box$ is an isomorphism in $\KRFrmC$.
\end{enumerate}
\end{lemma}

\begin{proof}
\eqref{lem:continuity SI(T) and B(box):item1}. Let $\Box=\RI(T)$. Then $\Box$ is a preframe homomorphism by \cref{thm: RI contravariant}. We define $\Diamond \colon \RI(B_2,S_2) \to \RI(B_1,S_1)$ by 
\[
\Diamond I = \{a \in B_1 \mid \exists b \in I : a \in L(\inv{T}[b]) \}.
\] 
We first show that $\Diamond$ is well defined. It is straightforward to see that $\Diamond I$ is an ideal of $B_1$. To see that $\Diamond I$ is a round ideal, let $a\in\Diamond I$. Then there is $b\in I$ with $a\in L(\inv{T}[b])$. Since $I$ is a round ideal, there is $d\in I$ with $b \S_2 d$. Because $T$ is continuous, there is $c\in \inv{T}[b]$ such that $c \in L(\inv{T}[d])$ (see \cref{rem:cont}\eqref{rem:cont:item1:cond1}). Therefore, $c \in \Diamond I$ since $d \in I$. Because $\inv{T}$ is compatible, from $b \mathrel{\inv{T}} c$ it follows that there is $c' \in \inv{T}[b]$ with $c' \S_1 c$. But then $a \le c'$ since $a \in L(\inv{T}[b])$. Thus, $a \le c' \S_1 c$, so $a \S_1 c$, and hence $\Diamond I$ is a round ideal.

We next show that $\Diamond$ preserves arbitrary joins. It is straightforward to see that $I \subseteq J$ implies $\Diamond I \subseteq \Diamond J$. Therefore, if $\{ I_\alpha \} \subseteq \RI(B_2,S_2)$, then $\bigvee \Diamond I_\alpha \subseteq \Diamond \left( \bigvee I_\alpha \right)$. For the reverse inclusion, let $x \in \Diamond \left( \bigvee I_\alpha \right)$. Then there is $b \in \bigvee I_\alpha$ with $x \in L(\inv{T}[b])$. Since $b \in \bigvee I_\alpha$, there exist $\alpha_1, \ldots, \alpha_n$ and $d_i \in I_{\alpha_i}$ for $i=1, \ldots, n$ such that $b \le d_1 \vee \dots \vee d_n$. Thus, $x \in L(\inv{T}[d_1 \vee \cdots \vee d_n])$. Because $I_{\alpha_i}$ is a round ideal for each $i$, it follows that there exist $e_i \in I_{\alpha_i}$ with $d_i \S_2 e_i$ for each $i$. By continuity of $T$, there exist $a_i \in \inv{T}[d_i]$ with $a_i \in L(\inv{T}[e_i])$ for each $i$. So $a_i \in \Diamond I_{\alpha_i}$ for each $i$ and $a_1 \vee \cdots \vee a_n \in \inv{T}[d_1 \vee \cdots \vee d_n]$. Since $x \in L(\inv{T}[d_1 \vee \cdots \vee d_n])$, it follows that $x \le a_1 \vee \cdots \vee a_n$. Consequently, $x \in \bigvee \Diamond I_\alpha$.

It is left to prove that $\Box I \cap \Diamond J \subseteq \Diamond(I \cap J)$ and $\Box(I \vee J) \subseteq \Box I \vee \Diamond J$ for all $I,J \in \RI(B_2,S_2)$. Let $x \in \Box I \cap \Diamond J$. Since $x \in \Box I = T^{-1}[I]$, there is $a \in I$ with $x \T a$. Because $x \in \Diamond J$, there is $b \in J$ with $x \in L(\inv{T}[b])$. We first show that $x \in L(\inv{T}[a \wedge b])$. If $e \in \inv{T}[a \wedge b]$, then $\neg e \T (\neg a \vee \neg b)$. Since $x \T a$, it follows that $(x \wedge \neg e) \T (a \wedge (\neg a \vee \neg b))$. So $(x \wedge \neg e) \T (a \wedge \neg b)$, and hence $(x \wedge \neg e) \T \neg b$. Therefore, $\neg x \vee e \in \inv{T}[b]$. Because $x \in L(\inv{T}[b])$, we have $x \le \neg x \vee e$, and so $x \le e$. Thus, $x \in L(\inv{T}[a \wedge b])$. Since $a \wedge b \in I \cap J$, we conclude that $x \in \Diamond(I \cap J)$.

Finally, 
let $x \in \Box(I \vee J)=T^{-1}[I \vee J]$. Then there is $y \in I \vee J$ with $x \T y$. Thus, there exist $a \in I$, $b \in J$ with $y \le a \vee b$. Since $I$ and $J$ are round ideals, there exist $a' \in I$, $b' \in J$ with $a \S_2 a'$ and $b \S_2 b'$. Because $\neg a' \S_2 \neg a$ and $b \S_2 b'$, the continuity of $T$ yields that there exist $c \in \inv{T}[\neg a']$ and $d \in \inv{T}[b]$ with $c \in L(\inv{T}[\neg a])$ and $d \in L(\inv{T}[b'])$. From $c \in \inv{T}[\neg a']$ 
it follows that $\neg c \T a'$, so $\neg c \in T^{-1}[I] = \Box I$. Since $d \in L(\inv{T}[b'])$ and $b' \in J$, we have $d \in \Diamond J$. Therefore, $\neg c \vee d \in \Box I \vee \Diamond J$. We prove that $x \le \neg c \vee d$, which is equivalent to $c \le \neg x \vee d$. We have $x \T (a \vee b)$ and $\neg d \T \neg b$ because $d \in \inv{T}[b]$. Therefore, $(x \wedge \neg d) \T ((a \vee b) \wedge \neg b)$, and so $(x \wedge \neg d) \T (a \wedge \neg b) \le a$. Thus, 
$\neg x \vee d \in \inv{T}[\neg a]$. Since $c \in L(\inv{T}[\neg a])$, we obtain $c \le \neg x \vee d$. Consequently, $x \in \Box I \vee \Diamond J$ because $x \le \neg c \vee d \in \Box I \vee \Diamond J$.

\eqref{lem:continuity SI(T) and B(box):item2}. Let $T = \bool(\Box)$. By \cref{lem:boolf compatible}, $T \colon \bool({L_2})\to\bool({L_1})$ is a morphism in $\SubSfive$. To see that it is continuous, let $b_1,b_2 \in \bool({L_1})$ with $b_1 \prec b_2$. Set $a = \neg \Box \neg b_2$. Then $a\in\bool(L_2)$. We show that $b_1 \mathrel{\inv{T}} a$ and $a \in L(\inv{T}[b_2])$. We have $\neg b_2 \prec \neg b_1$, so
$\Box \neg b_2 \prec \Box \neg b_1$ since $\Box$ preserves $\prec$ (see \cite[Lem.~3.6]{BBH15}). The definition of $\prec$ implies $\neg \neg \Box \neg b_2 \prec \Box \neg b_1$. Therefore, $\neg a \prec \Box \neg b_1$, which gives $\neg a \T  \neg b_1$. Thus, $b_1 \mathrel{\inv{T}} a$. If $x \in \inv{T}[b_2]$, then $\neg x \T \neg b_2$, so $\neg x \prec \Box \neg b_2$. Therefore, $a=\neg \Box \neg b_2 \prec x$, and hence $a \le x$. Thus, $a \in L(\inv{T}[b_2])$, and so $T$ is continuous.

\eqref{lem:continuity SI(T) and B(box):item3}. 
This is a consequence of a stronger result proved in \cref{lem:functionality SI(T) and B(box)}\eqref{lem:functionality SI(T) and B(box):item3} below.

\eqref{lem:continuity SI(T) and B(box):item4}. Since $\Box$ is an isomorphism in $\KRFrmP$, it is a poset isomorphism. Defining $\Diamond \coloneqq \Box$ then yields that $\Box$ is an isomorphism in $\KRFrmC$.
\end{proof}

As an immediate consequence of \cref{thm:dual-equivalence,lem:continuity SI(T) and B(box)} we obtain:

\begin{theorem} \label{thm: continuous case}
\mbox{}\begin{enumerate}[label=\normalfont(\arabic*), ref = \arabic*]
\item The dual equivalence between $\SubSfive$ and $\KRFrmP$ restricts to a dual equivalence between their wide subcategories $\SubSfiveCS$ and $\KRFrmC$.
\item The dual equivalence between $\devS$ and $\KRFrmP$ restricts to a dual equivalence between their wide subcategories $\devCS$ and $\KRFrmC$.
\end{enumerate}
\end{theorem}

We conclude this section by showing that $\devCS$ is dually isomorphic to $\devC$.
Let $(B_1,S_1)$ and $(B_2,S_2)$ be de Vries algebras. If $T \colon B_1 \to B_2$ is a morphism in $\devCS$, we define $\Box_T \colon B_2 \to B_1$ by $\Box_T b = \bigvee T^{-1}[b]$. Also, if $\Box \colon B_2 \to B_1$ is a morphism in $\devC$, we define $T_\Box \colon B_1 \to B_2$ by
\[
a \T_\Box b \iff \exists b'\in B_2 \ (a \S_1 \Box b'\mbox{ and }b' \S_2 b).
\]

\begin{lemma}\label{lem:continuity BoxT and TBox}
Let $(B_1,S_1)$ and $(B_2,S_2)$ be de Vries algebras.
\begin{enumerate}[label=\normalfont(\arabic*), ref = \arabic*]
\item\label{lem:continuity BoxT and TBox:item1} If $T \colon B_1 \to B_2$ is a morphism in $\devCS$, then $\Box_T \colon B_2 \to B_1$ is a morphism in $\devC$.
\item\label{lem:continuity BoxT and TBox:item2} If $\Box \colon B_2 \to B_1$ is a morphism in $\devC$, then $T_\Box \colon B_1 \to B_2$ is a morphism in $\devCS$.
\item\label{lem:continuity BoxT and TBox:item3} $\Box_{T_\Box}=\Box$.
\item\label{lem:continuity BoxT and TBox:item4} $T_{\Box_T}=T$.
\end{enumerate}
\end{lemma}

\begin{proof}
\eqref{lem:continuity BoxT and TBox:item1}. We first show that $\Box_T$ is de Vries multiplicative. It is obvious that $\Box_T 1=1$. Let $b_1 \S_2 b_2$ and $d_1 \S_2 d_2$. Since $T$ is continuous and $B_1$ is complete, by \cref{rem:cont}\eqref{rem:cont:item2:cond2}
\[
\left(\bigvee\T^{-1}[b_1]\right) \T b_2 \quad \mbox{and} \quad \left(\bigvee\T^{-1}[d_1]\right) \T d_2.
\] 
Therefore, $(\Box_T b_1 \wedge \Box_T d_1) \T (b_2 \wedge d_2)$. Since $T$ is compatible, there is $x \in B_1$ such that 
\[
(\Box_T b_1 \wedge \Box_T d_1) \S_1 x \T (b_2 \wedge d_2).
\]
Thus, $(\Box_T b_1 \wedge \Box_T d_1) \S_1 x \le \Box_T (b_2 \wedge d_2)$, and hence $(\Box_T b_1 \wedge \Box_T d_1) \S_1 \Box_T (b_2 \wedge d_2)$. Consequently, $\Box_T$ is de Vries multiplicative. To see that $\Box_T$ is lower continuous, let $x \in T^{-1}[b]$. Since $T$ is compatible, $x \T y \S_2 b$ for some $y \in B_2$. Therefore, $x \le \Box_T y$, and hence $\Box_T b = \bigvee \{ \Box_T y \mid y \S_2 b \}$. Thus, $\Box_T$ is a morphism in $\devC$.

\eqref{lem:continuity BoxT and TBox:item2}. That $0 \T_\Box 0$ is straightforward and that $1 \T_\Box 1$ follows from $\Box 1 = 1$. Since $\Box$ is lower continuous, it is order preserving (see \cite[Prop.~4.15(2)]{BBH15} and \cref{rem:box diamond upper lower}\eqref{rem:box diamond upper lower:item2}). 
Suppose $a,a' \T_\Box b$. Then there exist $b_1$ and $b_2$ such that $a \S_1 \Box b_1$, $b_1 \S_2 b$, $a' \S_1 \Box b_2$, and $b_2 \S_2 b$. From $a \S_1 \Box b_1$ and $a' \S_1 \Box b_2$ it follows that $(a \lor a') \S_1 (\Box b_1 \lor \Box b_2) \leq \Box(b_1 \lor b_2)$, and so $(a \lor a') \S_1 \Box(b_1 \lor b_2)$. Also, from $b_1 \S_2 b$ and $b_2 \S_2 b$ it follows that $(b_1 \lor b_2) \S_2 b$. Thus, $(a\vee a') \T_\Box b$. 
Next suppose $a \T_\Box b,b'$. Then there exist $b_1$ and $b_2$ such that $a \S_1 \Box b_1$, $b_1 \S_2 b$, $a \S_1 \Box b_2$, and $b_2 \S_2 b'$. From $a \S_1 \Box b_1$ and $a \S_1 \Box b_2$ it follows that $a \S_1 (\Box b_1 \land \Box b_2) = \Box (b_1 \land b_2)$ (see \cite[Prop.~4.15(2)]{BBH15} and \cref{rem:box diamond upper lower}\eqref{rem:box diamond upper lower:item2}). Also, from $b_1 \S_2 b$ and $b_2 \S_2 b'$ it follows that $(b_1 \land b_2) \S_2 (b \land b')$. Thus, $a \T_\Box (b \wedge b')$.
Finally, that $a \le a' \T_\Box b' \le b$ implies $a \T_\Box b$ is straightforward. This gives that $T_\Box$ is a subordination.

That $T_\Box \subseteq S_2 \circ T_\Box$ and $T_\Box \subseteq T_\Box \circ S_1$ follow from the fact that $S_2$ and $S_1$ satisfy \eqref{S7}. 
The reverse inclusions are obvious, so $S_2 \circ T_\Box = T_\Box = T_\Box \circ S_1$. This yields that $T_\Box$ is a compatible subordination.

It is left to prove that $T_\Box$ is continuous. Let $b_1 \S_2 b_2$. Then there is $y \in B_2$ with $b_1 \S_2 y \S_2 b_2$. Set $a=\Box b_1$. Since $a \S_1 \Box y$ and $y \S_2 b_2$, we have $a \T_\Box b_2$, so $a \in T_\Box^{-1}[b_2]$. Moreover, if $x \T_\Box b_1$, then there is $z \in B_2$ such that $x \S_1 \Box z$ and $z \S_2 b_1$. Therefore, $x \S_1 \Box b_1$, and so $x \S_1 a$. Thus, $a \in U(T_\Box^{-1}[b_1])$ by \eqref{S5}, and hence $T_\Box$ is continuous by \cref{rem:cont}\eqref{rem:cont:item1:cond2}. Consequently, $T_\Box$ is a morphism in $\devCS$.

\eqref{lem:continuity BoxT and TBox:item3}. We have
\[
\Box_{T_\Box} b = \bigvee T_\Box^{-1}[b] = \bigvee \{ a \mid \exists b'\in B_2 \ (a \S_1 \Box b' \mbox{ and } b' \S_2 b)\} = \bigvee \{ \Box b' \mid b' \S_2 b \} = \Box b,
\]
where the second to last equality follows from the facts that $S_2$ satisfies \eqref{S7}
and $b' \S_2 b$ implies $\Box b' \S_1 \Box b$, and the last equality from the lower continuity of $\Box$.

\eqref{lem:continuity BoxT and TBox:item4}. We have
\begin{align*}
a \mathrel{T_{\Box_T}} b & \iff \exists b'\in B_2 \ \left(a \S_1 \Box_T b' \mbox{ and } b' \S_2 b\right)\\
& \iff \exists b'\in B_2 \ \left(a \S_1 \bigvee T^{-1}[b'] \mbox{ and } b' \S_2 b\right).
\end{align*}
We show that the last condition is equivalent to $a \T b$.
Since $T$ is a morphism in $\devCS$ and $b' \S_2 b$, we have $\left(\bigvee T^{-1}[b'] \right) \T b$ by \cref{rem:cont}\eqref{rem:cont:item2:cond2}. Therefore, $a \S_1 \left(\bigvee T^{-1}[b'] \right) \T b$, and so $a \T b$.
Conversely, if $a \T b$, there are $a' \in B_1$ and $b' \in B_2$ such that $a \S_1 a' \T b' \S_2 b$. Thus, $a' \le \bigvee T^{-1}[b']$, and hence $a \S_1 \bigvee T^{-1}[b']$.
\end{proof}

As an immediate consequence of \cref{lem:continuity BoxT and TBox} we obtain:

\begin{theorem} \label{thm: continuous DeV}
$\devCS$ is dually isomorphic to $\devC$. 
\end{theorem}

Putting \cref{thm: continuous case,thm: continuous DeV} together yields the following analogue of the commutative diagram of equivalences and dual equivalences given at the end of \cref{sec:MacNeille}.
\[
\begin{tikzcd}[column sep=6ex, row sep=6ex]
& \SubSfiveCS \arrow[dl, "\RI"'] \arrow[dr, shift left=1ex, "\NI"] &  &\\
\KRFrmC \arrow[rr, "\bool"] & & \devCS \arrow[ul, hookrightarrow, shift left=0.5ex, "\F"] \arrow[r, leftrightarrow, "d"] & \devC
\end{tikzcd}
\]

\begin{remark}
As we pointed out in \cref{sec: prelims}, $\KRFrmC$ and $\devC$ are dually equivalent to $\KHausC$. Hence, $\SubSfiveCS$ and $\devCS$ are equivalent to $\KHausC$. 
The wide subcategories of $\StoneER$ and $\GleR$ that are equivalent to $\KHausC$ can be described as follows.

\label{def-StoneEC} \label{def-StoneGleC} Let $(X,E)$ be an $\mathsf{S5}$-subordination space. A morphism $R\colon X_1\to X_2$ in $\StoneER$ is {\em continuous} if $R^{-1}[U]$ is open for each $E_2$-saturated open $U \subseteq X_2$. Let $\StoneEC$ be the wide subcategory of $\StoneER$ whose morphisms are continuous morphisms in $\StoneER$ and define $\GleC$ similarly. Using \cref{t:correspondence continuous rel and sub} it is straightforward to see that the equivalence between $\StoneER$ and $\GleR$ described in 
\cref{rem:functors equivalences}\eqref{i:GleR StoneER equivalence}
restricts to an equivalence between $\StoneEC$ and $\GleC$. By \cite[Thm.~4.16]{BezhanishviliGabelaiaEtAl2019}, $\GleC$ is equivalent to $\KHausC$. Thus, each of $\KHausC$, $\StoneEC$, and $\GleC$ is equivalent or dually equivalent to each of the categories in the diagram above.
\end{remark}

\section{Functional subordinations}\label{sec:functional}

In this section we further restrict our attention to those wide subcategories of $\SubSfive$ and $\KRFrmP$ that encode continuous functions between compact Hausdorff spaces. The wide subcategories of $\SubSfive$ and $\StoneER$ equivalent to $\KHaus$ were described in \cite[Sec.~6]{ABC22a}, where it was shown that they are equivalent to the categories of maps in the allegories $\SubSfive$ and $\StoneER$.
This has resulted in the following notion:

\begin{definition} \cite[Def.~6.4]{ABC22a} 
\begin{enumerate}
\item Call a morphism $T \colon (B_1,S_1) \to (B_2,S_2)$ in $\SubSfive$ \emph{functional} if
\[
\inv{T} \circ T \subseteq S_1 \quad \mbox{and} \quad S_2 \subseteq T \circ \inv{T}.
\]
\item \label{def-subSfiveF}
Let $\SubSfiveF$ be 
the wide subcategory of $\SubSfive$ whose morphisms are functional morphisms, and define $\devF$ similarly.
\end{enumerate}
\end{definition}

\begin{remark}
If $T$ is functional, then $T$ is continuous. Indeed, let $b_1 \S_2 b_2$. Since $T$ is functional, $S_2 \subseteq T \circ \inv{T}$, so there exists $a \in B_1$ such that $b_1 \mathrel{\inv{T}} a$ and $a \T b_2$. Thus, $a \in \inv{T}[b_1]$. Moreover, if $a' \in \inv{T}[b_2]$, then $b_2 \mathrel{\inv{T}} a'$. Therefore, $a \T b_2 \mathrel{\inv{T}} a'$, so $a \S_1 a'$ because $\inv{T} \circ T \subseteq S_1$ by the functionality of $T$. Consequently, $T$ is continuous. Thus, $\SubSfiveF$ is a wide subcategory of $\SubSfiveCS$. Similarly, $\devF$ is a wide subcategory of $\devCS$.
\end{remark}

We now give a characterization of functional morphisms. For another characterization see \cite[Lem.~6.5]{ABC22a}.

\begin{lemma} \label{l:functional-alternative-def}
	Let $T \colon (B_1, S_1) \to (B_2, S_2)$ be a morphism in $\SubSfive$.
	The following conditions are equivalent.
	\begin{enumerate}[label=\normalfont(\arabic*), ref = \arabic*]
		\item \label{i:functional} $T$ is functional.
		\item \label{i:new-functional}
		The following hold for all $a \in B_1$ and $b_1, b_2, b_1', b_2' \in B_2$\textup:
		\begin{enumerate}[label=\normalfont(\alph*), ref = 2\alph*]
			\item \label{i:fun-0-ary} 
			If $a \T 0$, then $a = 0$.
			\item \label{i:fun-2-ary} 
			If $a \T (b_1 \lor b_2)$, $b_1 \S_2 b_1'$, and $b_2 \S_2 b_2'$, then there are $a_1, a_2 \in B_1$ such that $a \S_1 (a_1 \lor a_2)$, $a_1 \T b_1'$, and $a_2 \T b_2'$.
		\end{enumerate}
	\end{enumerate}
\end{lemma}

\begin{proof}
By \cite[Lem.~6.5(1)]{ABC22a}, $\inv{T} \circ T \subseteq S_1$ is equivalent to \eqref{i:fun-0-ary}. Therefore, it is sufficient to prove that, under these equivalent conditions, $S_2 \subseteq T \circ \inv{T}$ is equivalent to \eqref{i:fun-2-ary}.

	To prove that $S_2 \subseteq T \circ \inv{T}$ implies \eqref{i:fun-2-ary}, let $a \T (b_1 \lor b_2)$, $b_1 \S_2 b_1'$, and $b_2 \S_2 b_2'$.
	Since $S_2 \subseteq T \circ \inv{T}$, from $b_1 \S_2 b_1'$ and $b_2 \S_2 b_2'$ it follows that there are $a_1, a_2 \in B_1$ such that $b_1 \mathrel{\inv{T}} a_1 \T b_1'$ and $b_2 \mathrel{\inv{T}} a_2 \T b_2'$.
	Therefore, $a \T (b_1 \lor b_2) \mathrel{\inv{\T}} (a_1 \lor a_2)$.
	Since $\inv{T} \circ T \subseteq S_1$, 
	it follows that $a \S_1 (a_1 \lor a_2)$.
	
	To prove that \eqref{i:fun-2-ary} implies $S_2 \subseteq T \circ \inv{T}$,
	let $b_1,b_2 \in B_2$ be such that $b_1 \S_2 b_2$.
	By \eqref{S7}, there is $b \in B_2$ such that $b_1 \S_2 b \S_2 b_2$.
	We have $1 \T (\lnot b \lor b)$.
	By \eqref{S6}, $b_1 \S_2 b$ implies $\lnot b \S_2 \lnot b_1$.
	Thus, by \eqref{i:fun-2-ary}, there are $a_1, a_2 \in B_1$ such that $1 \S_1 (a_1 \lor a_2)$, $a_1 \T \lnot b_1$, and $a_2 \T b_2$.
	By \eqref{S5}, from $1 \S_1 (a_1 \lor a_2)$ it follows that $1 = a_1 \lor a_2$, so $\lnot a_1 \leq a_2$.
Since $a_1 \T \lnot b_1$, we have $b_1 \mathrel{\inv{T}} \neg a_1 \le a_2$, and hence $b_1 \mathrel{\inv{T}} a_2$. Because $b_1 \mathrel{\inv{T}} a_2 \T b_2$, it follows that $b_1 \mathrel{(T \circ \inv{T})} b_2$.
	Thus, $S_2 \subseteq T \circ \inv{T}$, completing the proof.	
\end{proof}

Our main goal in this section is to show that \cref{thm:dual-equivalence} restricts to yield the corresponding dual equivalences for $\SubSfiveF$ and $\devF$. For this we need \cref{lem:functionality SI(T) and B(box)}, which requires the following:

\begin{remark}\label{rem:inverse functional}
Let $T \colon (B_1, S_1) \to (B_2, S_2)$ be a morphism in $\SubSfive$. Since functional morphisms are maps in the allegory $\SubSfive$ \cite[Def.~6.4]{ABC22a}, it follows from \cite[p.~199]{FS90} that $T$ is an isomorphism iff $T$ and $\inv{T}$ are both functional, in which case $\inv{T}$ is the inverse of~$T$.
\end{remark}

\begin{lemma}\label{lem:functionality SI(T) and B(box)}
Let $(B_1,S_1),(B_2,S_2)$ be $\mathsf{S5}$-subordination algebras and $T \colon B_1 \to B_2$ be a morphism in $\SubSfive$. Let also $L_1,L_2$ be compact regular frames and $\Box \colon L_1 \to L_2$ a preframe homomorphism. 
\begin{enumerate}[label=\normalfont(\arabic*), ref = \arabic*]
\item\label{lem:functionality SI(T) and B(box):item1} If $T \colon B_1 \to B_2$ is a functional compatible subordination, then $\RI(T) \colon \RI(B_2) \to \RI(B_1)$ is a frame homomorphism.
\item\label{lem:functionality SI(T) and B(box):item2} If $\Box \colon L_1 \to L_2$ is a frame homomorphism, then $\mathfrak{B}(\Box) \colon \mathfrak{B}L_2 \to \mathfrak{B}L_1$ is functional.
\item\label{lem:functionality SI(T) and B(box):item3} If $T \colon B_1 \to B_2$ is an isomorphism in $\SubSfive$, then $T$ is an isomorphism in $\SubSfiveF$.
\item\label{lem:functionality SI(T) and B(box):item4} If $\Box \colon L_1 \to L_2$ is an isomorphism in $\KRFrmP$, then $\Box$ is an isomorphism in $\KRFrm$.
\end{enumerate}
\end{lemma}

\begin{proof}
	\eqref{lem:functionality SI(T) and B(box):item1}.
	Since $\RI(T)$ is a preframe homomorphism (see \cref{thm: RI contravariant}), it is sufficient to prove that it preserves bottom and binary joins.
	To see that $\RI(T)$ preserves bottom, it is enough to show that
	$T^{-1}[\{0\}] \subseteq \{0\}$, which follows from 
	\cref{l:functional-alternative-def}\eqref{i:fun-0-ary}.
	To see that $\RI(T)$ preserves binary joins,	
	let $I_1, I_2$ be round ideals of $B_2$.
	It is sufficient to prove that $T^{-1}[I_1 \lor I_2] \subseteq T^{-1}[I_1] \lor T^{-1}[I_2]$. 
	Let $a \in T^{-1}[I_1 \lor I_2]$.
	Then there are $b_1 \in I_1$, $b_2 \in I_2$ such that $a \T (b_1 \lor b_2)$.
	Since $I_1$ and $I_2$ are round ideals, there are $b_1' \in I_1$ and $b_2' \in I_2$ such that $b_1 \S_2 b_1'$ and $b_2 \S_2 b_2'$.
	By \cref{l:functional-alternative-def}\eqref{i:fun-2-ary}, there are $a_1, a_2 \in B_1$ such that $a \S_1 (a_1 \lor a_2)$, $a_1 \T b_1'$, and $a_2 \T b_2'$.
	Thus, $a \in T^{-1}[I_1] \lor T^{-1}[I_2]$.

	\eqref{lem:functionality SI(T) and B(box):item2}. We prove that $\mathfrak{B}(\Box)$ satisfies \cref{l:functional-alternative-def}\eqref{i:new-functional}.
	To see \eqref{i:fun-0-ary}, let $b \in \mathfrak{B}L_2$ be such that $b \mathrel{\mathfrak{B}(\Box)}0$, so $b \prec \Box 0$. Since $\Box$ is a frame homomorphism, $\Box 0 = 0$. Therefore, $b \prec 0$, and hence $b = 0$ by \eqref{S5}.
	To see \eqref{i:fun-2-ary}, 
	let $b \in \mathfrak{B}L_2$ and $a_1, a_2, a_1', a_2' \in \mathfrak{B}L_1$ be such that $b \mathrel{\mathfrak{B}(\Box)} (a_1 \lor a_2)$, $a_1 \prec a_1'$, and $a_2 \prec a_2'$.
	Then $b \prec \Box(a_1 \lor a_2)$. But $\Box(a_1 \lor a_2)=\Box a_1 \lor \Box a_2$ because $\Box$ is a frame homomorphism.
	Therefore, $b \prec \Box a_1 \lor \Box a_2$, and so there is $b' \in \mathfrak{B}(\Box)$ such that $b \prec b' \prec \Box a_1 \lor \Box a_2$.
	Set $b_1 = b' \land \Box a_1$ and $b_2 = b' \land \Box a_2$.
	We have $a_i \prec a_i'$ implies $\Box a_i \prec \Box a_i'$ for $i \in \{1, 2\}$.
	Thus, $b_i = b' \land \Box a_i \leq \Box a_i \prec \Box a_i'$, so $b_i \prec \Box a_i'$, and hence $b_i \mathrel{\mathfrak{B}(\Box)} a_i'$.
	Moreover, from $b \prec b'$ and $b \prec \Box a_1 \lor \Box a_2$ it follows that 
	\[
	b \prec b' \land (\Box a_1 \lor \Box a_2) = (b' \land \Box a_1) \lor (b' \land \Box a_2) = b_1 \lor b_2.
	\]
	This proves \eqref{i:fun-2-ary}.

	\eqref{lem:functionality SI(T) and B(box):item3}.
	This follows from \cref{rem:inverse functional}.
	
	\eqref{lem:functionality SI(T) and B(box):item4}.
	In both $\KRFrmP$ and $\KRFrm$ isomorphisms are order-isomorphisms.
\end{proof}

From \cref{thm:dual-equivalence,lem:functionality SI(T) and B(box)} we obtain:

\begin{theorem} \label{t:restriction-functional}
$ $
\begin{enumerate}[label=\normalfont(\arabic*), ref = \arabic*]
\item The dual equivalence between $\SubSfive$ and $\KRFrmP$ restricts to a dual equivalence between their wide subcategories $\SubSfiveF$ and $\KRFrm$.
\item The dual equivalence between $\devS$ and $\KRFrmP$ restricts to a dual equivalence between their wide subcategories $\devF$ and $\KRFrm$.
\end{enumerate}
\end{theorem}

In addition, we have:

\begin{theorem}[{\cite[Thm.~6.18]{ABC22a}}]\label{thm:dev and devf iso}
$\dev$ and $\devF$ are dually isomorphic. 
\end{theorem}

Consequently, we arrive at the following analogue of the commutative diagram of equivalences and dual equivalences given at the end of \cref{sec: cont subs}. 
\[
\begin{tikzcd}[column sep=6ex, row sep=6ex]
& \SubSfiveF \arrow[dl, "\RI"'] \arrow[dr, shift left=1ex, "\NI"] &  &\\
\KRFrm \arrow[rr, "\bool"] & & \devF \arrow[ul, hookrightarrow, shift left=0.5ex, "\F"] \arrow[r, leftrightarrow, "d"] & \dev
\end{tikzcd}
\]

\begin{remark} \label{def-StoneEF} \label{def-Gle}
We recall from \cite[Def.~6.1]{ABC22a} that $\StoneEF$ is the wide subcategory of $\StoneER$ whose morphisms $R\colon (X_1,E_1)\to (X_2,E_2)$ satisfy $E_1\subseteq \conv{R}\circ R$ and $R\circ\conv{R}\subseteq E_2$. We call such morphisms \emph{functional} and define $\Gle$ similarly. By \cite[Thm.~6.9]{ABC22a}, the categories $\SubSfiveF$, $\devF$, $\StoneEf$, $\Gle$, and $\KHaus$  are equivalent. Thus, each of these is equivalent or dually equivalent to the categories in the above diagram. 
\end{remark}

We thus arrive at the following diagram, in which empty boxes of the diagram in \cref{fig:diagram1} are filled. The number under each double arrow indicates the corresponding statement in the body of the paper.

\begin{figure}[h!]
\[
\begin{tikzcd}[column sep=0.90cm]
\SubSfive \arrow[r, leftrightarrow, "\ref{thm: NI equivalence}"'] & \devS \arrow[rr, leftrightarrow, "d", "\ref{thm:dual-equivalence}"'] & & \KRFrmP \arrow[r, leftrightarrow, "d", "\ref{thm:isbell-gen}"'] & \KHausR \arrow[r, leftrightarrow, "\ref{thm:equiv-ABC}"'] & \StoneER \arrow[r, leftrightarrow, "\ref{thm:equiv-ABC}"'] & \GleR \\
\SubSfiveCS \arrow[u, hookrightarrow] \arrow[r, leftrightarrow, "\ref{thm: continuous case}"'] & \devCS \arrow[u, hookrightarrow] \arrow[r, leftrightarrow, "d", "\ref{thm: continuous DeV}"'] & \devC \arrow[r, leftrightarrow, "{\ref{thm:KRFrmC-dual-KHausC},\ref{t:continuous}}"'] & \KRFrmC \arrow[u, hookrightarrow] \arrow[r, leftrightarrow, "d", "\ref{thm:KRFrmC-dual-KHausC}"'] & \KHausC \arrow[u, hookrightarrow] \arrow[r, leftrightarrow, "\ref{def-StoneEC}"'] & \StoneEC \arrow[u, hookrightarrow] \arrow[r, leftrightarrow, "\ref{def-StoneEC}"'] & \GleC \arrow[u, hookrightarrow] \\
\SubSfiveF \arrow[u, hookrightarrow] \arrow[r, leftrightarrow, "\ref{t:restriction-functional}"'] & \devF \arrow[u, hookrightarrow] \arrow[r, leftrightarrow, "d", "\ref{thm:dev and devf iso}"'] & \dev \arrow[u, hookrightarrow] \arrow[r, leftrightarrow, "{\ref{t:Isbell},\ref{t:de-Vries-duality}}"'] & \KRFrm \arrow[u, hookrightarrow] \arrow[r, leftrightarrow, "d", "\ref{t:Isbell}"'] & \KHaus \arrow[u, hookrightarrow] \arrow[r, leftrightarrow, "\ref{def-StoneEF}"'] & \StoneEF \arrow[u, hookrightarrow] \arrow[r, leftrightarrow, "\ref{def-StoneEF}"'] & \Gle \arrow[u, hookrightarrow]
\end{tikzcd}
\]
\caption{}\label{fig:diagram2}
\end{figure}

For the reader's convenience we also list all the categories involved in the diagram.

\begin{table}[H]
	\begin{center}	
		\begin{tabular}{ | p{\widthof{\textbf{Category}}} | p{\widthof{Compact Hausdorff spaces}} | p{\widthof{Continuous compatible closed relations}} |}
			\hline
			\textbf{Category} & \textbf{Objects} & \textbf{Morphisms}\\
			\hline
			$\SubSfive$ & $\mathsf{S5}$-subordination algebras & Compatible subordinations\\
			$\SubSfiveCS$ & $\mathsf{S5}$-subordination algebras & Continuous compatible subordinations\\
			$\SubSfiveF$ & $\mathsf{S5}$-subordination algebras & Functional compatible subordinations\\			$\devS$ & De Vries algebras & Compatible subordinations\\
			$\devCS$ & De Vries algebras & Continuous compatible subordinations\\
			$\devF$ & De Vries algebras & Functional compatible subordinations\\
			$\devC$ & De Vries algebras & Lower continuous de Vries mult. maps\\
			$\dev$ & De Vries algebras & De Vries morphisms\\
			\hline
		\end{tabular}
	\end{center}
	\caption{Categories of subordination algebras.}\label{table1}
\end{table}
\begin{table}[H]
	\begin{center}	
		\begin{tabular}{ | p{\widthof{\textbf{Category}}} | p{\widthof{Compact Hausdorff spaces}} | p{\widthof{Continuous compatible closed relations}}|}
			\hline
			\textbf{Category} & \textbf{Objects} & \textbf{Morphisms}\\ 
			\hline
			$\KRFrmP$ & Compact regular frames & Preframe homomorphisms\\
			$\KRFrmC$ & Compact regular frames & Continuous preframe homomorphisms\\
			$\KRFrm$ & Compact regular frames & Frame homomorphisms\\
			\hline
		\end{tabular}
	\end{center}
	\caption{Categories of compact regular frames.}\label{table2}
\end{table}
\begin{table}[H]
	\begin{center}	
		\begin{tabular}{ | p{\widthof{\textbf{Category}}} | p{\widthof{Compact Hausdorff spaces}} | p{\widthof{Continuous compatible closed relations}} |}
			\hline
			\textbf{Category} & \textbf{Objects} & \textbf{Morphisms}\\
			\hline
			$\KHausR$ & Compact Hausdorff spaces & Closed relations\\
			$\KHausC$ & Compact Hausdorff spaces & Continuous relations\\
			$\KHaus$ & Compact Hausdorff spaces & Continuous functions\\
			\hline
		\end{tabular}
	\end{center}
	\caption{Categories of compact Hausdorff spaces.}\label{table3}
\end{table}
\begin{table}[H]
	\begin{center}	
		\begin{tabular}{ | p{\widthof{\textbf{Category}}} | p{\widthof{Compact Hausdorff spaces}} | p{\widthof{Continuous compatible closed relations}} |}
			\hline
			\textbf{Category} & \textbf{Objects} & \textbf{Morphisms}\\
			\hline
			$\StoneER$ & $\mathsf{S5}$-subordination spaces & Compatible closed relations\\
			$\StoneEC$ & $\mathsf{S5}$-subordination spaces & Continuous compatible closed relations\\
			$\StoneEF$ & $\mathsf{S5}$-subordination spaces & Functional compatible closed relations\\
			$\GleR$ & Gleason spaces & Compatible closed relations\\
			$\GleC$ & Gleason spaces & Continuous compatible closed relations\\
			$\Gle$ & Gleason spaces & Functional compatible closed relations\\
			\hline
		\end{tabular}
	\end{center}
	\caption{Categories of subordination spaces.}\label{table4}
\end{table}

\section{Dual descriptions of the completions}\label{sec: dually}

In this final section we give dual descriptions of the round ideal and MacNeille completions of $\mathsf{S5}$-subordination algebras. 

Recall that if $B$ is a boolean algebra and $X$ is the Stone space of $B$, then the isomorphism $\varphi \colon B \to \Clop(X)$ is given by the Stone map $\varphi(a) = \{ x \in X \mid a \in x \}$. This isomorphism induces an 
order-isomorphism $\Phi$ between the frame of ideals of $B$ and the frame of open subsets of $X$, as well as an order-isomorphism $\Psi$ between the frame of filters of $B$ and the frame of closed subsets of $X$ ordered by reverse inclusion (see, e.g., \cite[Thm.~33]{GH09}). The isomorphisms are defined as follows: 
\begin{equation*}
\Phi(I)=\bigcup\{\varphi(a)\mid a\in I\} \quad \mbox{and} \quad \Psi(F)=\bigcap\{\varphi(a)\mid a\in F\}.
\end{equation*}
It belongs to folklore that for an ideal $I$ and filter $F$ of $B$, we have 
\begin{equation}\label{eq:phi psi negation}
\begin{array}{lll}
& \Phi(\neg F)=\Psi(F)^c, \qquad & \Phi(L(F))=\int\, (\Psi(F)), \\
& \Psi(\neg I)=\Phi(I)^c, \qquad & \Psi(U(I))=\cl\, (\Phi(I)).
\end{array}
\end{equation}
For the reader's convenience, we give a proof of $\Psi(U(I))=\cl\, (\Phi(I))$. The other three equalities are proved similarly.
Since $b \in U(I)$ iff $\varphi(a) \subseteq \varphi(b)$ for each $a \in I$, we have
\begin{align*}
	\Psi(U(I)) & = \bigcap \{ \varphi(b) \mid b \in U(I) \} = \bigcap \{ \varphi(b) \mid \Phi(I) \subseteq \varphi(b) \} =\cl (\Phi(I)),
\end{align*}
where the last equality follows from the fact that $X$ is a Stone space, hence the closure of a set is the intersection of the clopen sets containing it.

Let $(B,S)$ be an $\mathsf{S5}$-subordination algebra. We recall from \cref{rem:functors equivalences}\eqref{i:functor-Ult} that the $\mathsf{S5}$-subordination space of $(B,S)$ is $(X,R_S)$ where $X$ is the Stone space of $B$ and $R_S$ is given by $x \R_S y$ iff $S[x] \subseteq y$. For simplicity, we write $(X,R)$ instead of $(X,R_S)$.

\begin{lemma}\label{lem:S and * dually}
Let $(B,S)$ be an $\mathsf{S5}$-subordination algebra and $(X,R)$ its $\mathsf{S5}$-subordination space.
\begin{enumerate}[label=\normalfont(\arabic*), ref = \arabic*]
\item \label{i:ideal} If $I$ is an ideal of $B$, then $\Phi(S^{-1}[I])=\Box_{R} \Phi(I)$. 
\item \label{i:filter} If $F$ is a filter of $B$, then $\Psi(S[F])=R[\Psi(F)]$.
\end{enumerate}
\end{lemma}

\begin{proof}
\eqref{i:ideal}. 
We have
\begin{align*}
\Phi(S^{-1}[I]) &= \bigcup \{ \varphi(a) \mid a \in S^{-1}[I] \} = \bigcup \{ \varphi(a) \mid \exists \, b \in I : a \S b \}\\
&=\bigcup \{ \varphi(a) \mid \exists \, b \in I : R[\varphi(a)] \subseteq \varphi(b) \} = \bigcup \{ \varphi(a) \mid R[\varphi(a)] \subseteq \Phi(I) \}\\
&= \bigcup \{ \varphi(a) \mid \varphi(a) \subseteq\Box_R \Phi(I) \} = \Box_R \Phi(I),
\end{align*}
where the third equality follows from the fact that $a \S b$ iff $R[\varphi(a)] \subseteq \varphi(b)$ (see, e.g., \cite[Lem.~2.20]{BBSV17}); the fourth from the fact that $R[\varphi(a)]$ is closed, hence compact in $X$; and the last from the fact that $\Box_R \Phi(I)$ is open and $\{ \varphi(a) \mid a \in B\}$ forms a basis for $X$.
	
\eqref{i:filter}. 
We have\textup:
\begin{align*}
\Psi(S[F]) & =(\Phi(\neg S[F]))^c & \text{(by~\eqref{eq:phi psi negation})}\\
& =(\Phi(S^{-1}[\neg F]))^c & \text{(by~\cref{lem:operators and negations})}\\
& =(\Box_R\Phi(\neg F))^c & \text{(by item~\eqref{i:ideal})}\\
& =(\Box_R(\Psi(F)^c))^c & \text{(by~\eqref{eq:phi psi negation})}\\
& =R[\Psi(F)] & \text{(by~\cref{rem:box}\eqref{rem:box:item2}).} & \qedhere
\end{align*}
\end{proof}

We recall from the introduction that $\O(X)$ denotes the frame of open subsets of a topological space $X$. 
Since the set of  $R$-saturated open  subsets of an $\mathsf{S5}$-subordination space $(X,R)$ forms a subframe of $\O(X)$, it is a frame.

\begin{definition}
For an $\mathsf{S5}$-subordination space $\mathbf X=(X,R)$ let $\O_R(\mathbf X)$ be the frame of  $R$-saturated open  subsets of $X$.  
\end{definition}

\begin{lemma}\label{lem: SI(B) iso to ORS(Ult(B))}
Let $\mathbf B=(B,S)$ be an $\mathsf{S5}$-subordination algebra and $\mathbf X=(X,R)$ its $\mathsf{S5}$-subordination space. 
An ideal $I$ of $B$ is a round ideal iff $\Phi(I)$ is an  $R$-saturated open subset of $X$.
Therefore, $\RI(\mathbf B)$ is isomorphic to $\O_R(\mathbf X)$.
\end{lemma}

\begin{proof}
We have that $I$ is a round ideal iff $I=S^{-1}[I]$. Since $\Phi$ is an isomorphism, it follows from \cref{lem:S and * dually}\eqref{i:ideal} that $I$ is a round ideal iff $\Phi(I) = \Box_R \Phi(I)$. Therefore, $I$ is a round ideal iff $\Phi(I)$ is $R$-saturated. 
Thus, the restriction of $\Phi$ is an isomorphism from $\RI(\mathbf B)$ to $\O_R(\mathbf X)$.
\end{proof}

Let $\mathbf{X}=(X,R)$ be an $\mathsf{S5}$-subordination space and $\pi \colon X \to  X/R$ the quotient map given by $\pi(x)=[x]$. It is well known that $\pi$ lifts to an isomorphism between $\O(X/R)$ and $\O_R(\mathbf X)$ (see, e.g., \cite[Prop.~2.4.3]{Eng89}).
This together with \cref{lem: SI(B) iso to ORS(Ult(B))} yields the following result, which by Isbell duality gives an alternative proof of \cref{thm:Round ideals frame}\eqref{thm:Round ideals frame:item4}.

\begin{theorem}\label{thm:RI(B) iso O(X/R)}
Let $\mathbf{B}=(B,S)$ be an $\mathsf{S5}$-subordination algebra and $\mathbf X=(X,R)$ its subordination space. Then
$\RI(\mathbf{B})$ is isomorphic to $\O(X/R)$.
\end{theorem}

We recall that the MacNeille completion of a boolean algebra $B$ is isomorphic to $\RO(X)$ where $X$ is the Stone space of $B$ (see, e.g.,~\cite[Thm.~40]{GH09}).
We will generalize this result to the setting of $\mathsf{S5}$-subordination algebras. 
Since regular opens are fixpoints of $\int\, \cl\colon\O(X)\to\O(X)$, we introduce the notion of an $R$-regular open subset of an $\mathsf{S5}$-subordination space $(X,R)$ by replacing $\int$ with $\Box_R \, \int$ and $\cl$ with $R \, \cl$. 

\begin{definition}\label{def: R-regular open}
Let $\mathbf X=(X,R)$ be an $\mathsf{S5}$-subordination space.
We say that an  $R$-saturated open  subset of $X$ is {\em $R$-regular open} if it is a fixpoint of $ \Box_{R}\,\int\, R\,\cl \colon \O_R(\mathbf X)\to\O_R(\mathbf X)$. Let $\RO_R(\mathbf X)$ be the poset of $R$-regular open subsets  of $X$. 
\end{definition}

\begin{lemma} \label{lem: R-regular open}
Let $\mathbf X=(X,R)$ be an $\mathsf{S5}$-subordination space. 
Equip $\RO_R(\mathbf X)$ with the relation $\prec$ given by 
\[
U \prec V \iff R[ \cl (U)] \subseteq V.
\]
Then $\RO_R(\mathbf X)$ is a de Vries algebra isomorphic to $\RO(X/R)$.
\end{lemma}

\begin{proof}
As we pointed out in the paragraph before \cref{thm:RI(B) iso O(X/R)}, $\pi \colon X \to X/R$ lifts to an isomorphism $f \colon \O_R(X) \to \O(X/R)$ given by $f(U)=\pi[U]$. 
We show that for each $U \in \O_R(X)$ we have
\[
U \in \RO_R(X) \iff \pi[U] \in \RO(X/R).
\]
On the one hand,
\begin{align*}
U \in \RO_R(X) \iff U=\Box_R(\int (R[\cl(U)])) \iff \pi[U]=\pi[\Box_R(\int (R[\cl(U)]))].
\end{align*}
On the other hand, 
\begin{align*}
\pi[U] \in \RO(X/R) \iff \pi[U]=\int (\cl(\pi[U])).
\end{align*}
Therefore, it is enough to prove that 
\[
\pi[\Box_R(\int (R[\cl(U)]))]=\int (\cl(\pi[U])).
\]
Since $\pi \colon X \to X/R$ is a quotient map and $X/R$ is compact Hausdorff, $\pi$ is a closed map. Thus, for each $R$-saturated subset $G$ of $X$ we have
\begin{align}\label{eq:frclA=clfA}
\pi[R[\cl(G)]]=\pi[\cl(G)]=\cl(\pi[G]).
\end{align}
Moreover, since $G$ is $R$-saturated,
\begin{equation}\label{eq:pi[Ac]=pi[A]c}
\pi[G^c]=\pi[G]^c.
\end{equation}
Therefore, if $H$ is an $R$-saturated subset of $X$, then 
\begin{align*}\label{eq:fBoxintA=intfA}
\pi[\Box_R(\int (H))] & =\pi[R[\cl(H^c)]^c] & \\
& = \pi[R[\cl(H^c)]]^c & \text{(by \eqref{eq:pi[Ac]=pi[A]c})}\\
& = \cl(\pi[H^c])^c & \text{(by \eqref{eq:frclA=clfA})}\\
& = \int (\pi[H^c]^c) & \\
& = \int(\pi[H]) & \text{(by \eqref{eq:pi[Ac]=pi[A]c}).}
\end{align*}
This equation together with \eqref{eq:frclA=clfA} yields 
\[
\pi[\Box_R(\int (R[\cl(U)]))]=\int(\pi[R[\cl(U)]]) = \int (\cl(\pi[U])).
\]
Thus, $f$ restricts to a poset isomorphism and hence a boolean isomorphism between $\RO_R(X)$ and $\RO(X/R)$. By \eqref{eq:frclA=clfA}, $f$ also preserves and reflects the relation:
\begin{align*}
U \prec V & \iff R[\cl(U)] \subseteq V \iff \pi[R[\cl(U)]] \subseteq \pi[V]\\
& \iff \cl(\pi[U]) \subseteq \pi[V] \iff \pi[U] \prec \pi[V].
\end{align*}
Therefore, $f$ is a structure-preserving bijection, hence an isomorphism of de Vries algebras by \cite[Prop.~1.5.5]{deV62}.
\end{proof}

\begin{proposition}\label{lem: SI and NI}
Let $\mathbf B=(B,S)$ be an $\mathsf{S5}$-subordination algebra and $\mathbf X=(X,R)$ its $\mathsf{S5}$-subordination space. For a round ideal $I$ of $\mathbf{B}$, we have\textup:
\begin{enumerate}[label=\normalfont(\arabic*), ref = \arabic*]
\item \label{i:round ideal} $\Phi(I^*)=\Box_R \int (\Phi(I)^c)$.
\item \label{i:normal round ideal} $\Phi(I^{**})=\Box_R\int(R[\cl\, \Phi(I)])$.
\item \label{i:iso normal round ideals and R-regular} $I$ is a normal round ideal iff $\Phi(I)$ is an $R$-regular open subset.
\end{enumerate}
Consequently, $\NI(\mathbf B)$ is isomorphic to $\RO_R(\mathbf X)$.
\end{proposition}

\begin{proof}
\eqref{i:round ideal}.  We have
\begin{align*}
\Phi(I^*) & = \Phi(\neg S[U(I)]) & \text{(by \cref{thm:Round ideals frame}\eqref{thm:Round ideals frame:item2})}\\
& = (\Psi(S[U(I)]))^c & \text{(by \eqref{eq:phi psi negation})}\\
& = (R[\Psi(U(I))])^c & \text{(by \cref{lem:S and * dually}\eqref{i:filter})}\\
& = (R[\cl\,\Phi(I)])^c & \text{(by \eqref{eq:phi psi negation})}\\
& = \Box_R \int (\Phi(I)^c), &
\end{align*}
where the last equality follows from the fact that $\cl\, U = (\int(U^c))^c$ for each $U \subseteq X$.

\eqref{i:normal round ideal}. By the proof of item~\eqref{i:round ideal}, if $I$ is a round ideal, then 
\[
\Phi(I^*)=(R[\cl\,\Phi(I)])^c=\Box_R \int (\Phi(I)^c).
\]
Thus,
\begin{equation*}
\Phi(I^{**}) = \Box_R \int (\Phi(I^*)^c) = \Box_R \int (((R[\cl\,\Phi(I)])^c)^c) = \Box_R\int(R[\cl\, \Phi(I)]).
\end{equation*}

\eqref{i:iso normal round ideals and R-regular}. Since $I$ is normal iff $I=I^{**}$, this follows from item~\eqref{i:normal round ideal} and \cref{def: R-regular open}.

Finally, since $\Phi$ is an order-isomorphism, its restriction is an isomorphism of the boolean algebras $\NI(\mathbf B)$ and $\RO_R(\mathbf X)$. 
Moreover, if $I,J \in \NI(\mathbf B)$, then
\begin{align*}
I \prec J & \iff I^* \vee J = B& \\
&  \iff \Phi(I^* \vee J)=X  & \\
& \iff \Phi(I^*) \cup \Phi(J)=X &\\
& \iff R[ \cl\, \Phi(I)]^c \cup \Phi(J)=X & \text{(by the proof of item~\eqref{i:round ideal})}\\
& \iff R[ \cl\, \Phi(I)] \subseteq \Phi(J) & \\
& \iff \Phi(I) \prec \Phi(J).   & 
\end{align*}
Therefore, $\Phi$ is an isomorphism of de Vries algebras.
\end{proof}

Combining \cref{lem: R-regular open,lem: SI and NI} yields the following result, which gives an alternative proof of 
\cref{prop:B in subs5 then NI(B) in devs}.

\begin{theorem}
Let $\mathbf{B}=(B,S)$ be an $\mathsf{S5}$-subordination algebra and $\mathbf X=(X,R)$ its $\mathsf{S5}$-subordination space. Then
$\NI(\mathbf{B})$ is isomorphic to $\RO(X/R)$.
\end{theorem}

\section*{Acknowledgements}
We are very thankful to the referee for careful reading and useful comments which have improved the paper considerably.

Marco Abbadini was supported by the Italian Ministry of University and Research through the PRIN project n.\ 20173WKCM5 \emph{Theory and applications of resource sensitive logics} and by UK Research and Innovation (UKRI) under the UK government’s Horizon Europe funding guarantee (grant number EP/Y015029/1, Project ``DCPOS'').

Luca Carai was supported by the Italian Ministry of University and Research through the PRIN project n.\ 20173WKCM5 \emph{Theory and applications of resource sensitive logics} and by the Spanish Ministry of Science and Innovation (MCIN) and the European Union through the Juan de la Cierva-Formaci\'on 2021 programme (FJC2021-046977-I).


\begin{thebibliography}{BGHJ19}

\bibitem[ABC23]{ABC22a}
M.~Abbadini, G.~Bezhanishvili, and L.~Carai.
\newblock A generalization of de {V}ries duality to closed relations between
  compact {H}ausdorff spaces.
\newblock {\em Topology Appl.}, 337:Paper No. 108641, 2023.

\bibitem[Ban89]{Ban89}
B.~Banaschewski.
\newblock Universal zero-dimensional compactifications.
\newblock In {\em Categorical topology and its relation to analysis, algebra
  and combinatorics ({P}rague, 1988)}, pages 257--269. World Sci. Publ.,
  Teaneck, NJ, 1989.

\bibitem[BBH15]{BBH15}
G.~Bezhanishvili, N.~Bezhanishvili, and J.~Harding.
\newblock Modal compact {H}ausdorff spaces.
\newblock {\em J. Logic Comput.}, 25(1):1--35, 2015.

\bibitem[BBSV17]{BBSV17}
G.~Bezhanishvili, N.~Bezhanishvili, S.~Sourabh, and Y.~Venema.
\newblock Irreducible equivalence relations, {G}leason spaces, and de {V}ries
  duality.
\newblock {\em Appl. Categ. Structures}, 25(3):381--401, 2017.

\bibitem[BBSV19]{BBSV19}
G.~Bezhanishvili, N.~Bezhanishvili, T.~Santoli, and Y.~Venema.
\newblock A strict implication calculus for compact {H}ausdorff spaces.
\newblock {\em Ann. Pure Appl. Logic}, 170(11), 2019.
\newblock 102714.

\bibitem[Bez10]{Bez10}
G.~Bezhanishvili.
\newblock Stone duality and {G}leason covers through de {V}ries duality.
\newblock {\em Topology Appl.}, 157(6):1064--1080, 2010.

\bibitem[Bez12]{Bez12}
G.~Bezhanishvili.
\newblock De {V}ries algebras and compact regular frames.
\newblock {\em Appl. Categ. Structures}, 20(6):569--582, 2012.

\bibitem[BGHJ19]{BezhanishviliGabelaiaEtAl2019}
G.~{Bezhanishvili}, D.~{Gabelaia}, J.~{Harding}, and M.~{Jibladze}.
\newblock {Compact Hausdorff spaces with relations and Gleason spaces}.
\newblock {\em {Appl. Categ. Structures}}, 27(6):663--686, 2019.

\bibitem[BH14]{BH14}
G.~Bezhanishvili and J.~Harding.
\newblock Proximity frames and regularization.
\newblock {\em Appl. Categ. Structures}, 22(1):43--78, 2014.

\bibitem[BP96]{BP96}
B.~Banaschewski and A.~Pultr.
\newblock Booleanization.
\newblock {\em Cah. Topol. G\'{e}om. Diff\'{e}r. Cat\'{e}g.}, 37(1):41--60,
  1996.

\bibitem[Cel01]{Celani2001}
S.~A. Celani.
\newblock {Quasi-modal algebras}.
\newblock {\em {Math. Bohem.}}, 126(4):721--736, 2001.

\bibitem[Cel18]{Cel18}
S.~A. Celani.
\newblock Quasi-semi-homomorphisms and generalized proximity relations between
  {B}oolean algebras.
\newblock {\em Miskolc Math. Notes}, 19(1):171--189, 2018.

\bibitem[dV62]{deV62}
H.~de~Vries.
\newblock {\em Compact spaces and compactifications. {A}n algebraic approach}.
\newblock PhD thesis, University of Amsterdam, 1962.

\bibitem[DV06]{DV06}
G.~Dimov and D.~Vakarelov.
\newblock Topological representation of precontact algebras.
\newblock In {\em Relational Methods in Computer Science}, volume 3929 of {\em
  Lecture Notes in Comput. Sci.}, pages 1--16. Springer, Berlin, 2006.

\bibitem[DV07]{DuentschVakarelov2007}
I.~{D\"untsch} and D.~{Vakarelov}.
\newblock {Region-based theory of discrete spaces: A proximity approach}.
\newblock {\em {Ann. Math. Artif. Intell.}}, 49(1-4):5--14, 2007.

\bibitem[Eng89]{Eng89}
R.~Engelking.
\newblock {\em General topology}, volume~6 of {\em Sigma Series in Pure
  Mathematics}.
\newblock Heldermann Verlag, Berlin, second edition, 1989.

\bibitem[FS90]{FS90}
P.~J. Freyd and A.~Scedrov.
\newblock {\em Categories, allegories}, volume~39 of {\em North-Holland
  Mathematical Library}.
\newblock North-Holland Publishing Co., Amsterdam, 1990.

\bibitem[GH09]{GH09}
S.~Givant and P.~Halmos.
\newblock {\em Introduction to {B}oolean algebras}.
\newblock Undergraduate Texts in Mathematics. Springer, New York, 2009.

\bibitem[GK81]{GK81}
G.~Gierz and K.~Keimel.
\newblock Continuous ideal completions and compactifications.
\newblock In B.~Banaschewski and R.-E. Hoffmann, editors, {\em Continuous
  Lattices}, pages 97--124, Berlin, Heidelberg, 1981. Springer Berlin
  Heidelberg.

\bibitem[Gr{\"{a}}78]{Gra78}
G.~Gr{\"{a}}tzer.
\newblock {\em General lattice theory}, volume Band 52 of {\em Lehrb\"ucher und
  Monographien aus dem Gebiete der Exakten Wissenschaften, Mathematische
  Reihe}.
\newblock Birkh\"auser Verlag, Basel-Stuttgart, 1978.

\bibitem[Isb72]{Isb72}
J.~Isbell.
\newblock Atomless parts of spaces.
\newblock {\em Math. Scand.}, 31:5--32, 1972.

\bibitem[JKM01]{JKM01}
A.~{Jung}, M.~{Kegelmann}, and M.~A. {Moshier}.
\newblock {Stably compact spaces and closed relations}.
\newblock In {\em MFPS 2001. Papers from the 17th Conference on the
  Mathematical Foundations of Programming Semantics, Aarhus University, Aarhus,
  Denmark, May 23--26, 2001}, pages 209--231. Amsterdam: Elsevier, 2001.

\bibitem[Joh82]{Joh82}
P.~T. Johnstone.
\newblock {\em Stone spaces}, volume~3 of {\em Cambridge Studies in Advanced
  Mathematics}.
\newblock Cambridge University Press, Cambridge, 1982.

\bibitem[KMJ23]{KurzMoshierEtAl2023}
A.~Kurz, A.~Moshier, and A.~Jung.
\newblock Stone duality for relations.
\newblock In A.~Palmigiano and M.~Sadrzadeh, editors, {\em Samson Abramsky on
  Logic and Structure in Computer Science and Beyond}, pages 159--215. Springer
  International Publishing, Cham, 2023.

\bibitem[Mee66]{Mee66}
K.~N. Meenakshi.
\newblock Proximity structures in {B}oolean algebras.
\newblock {\em Acta Sci. Math. (Szeged)}, 27:85--92, 1966.

\bibitem[PP12]{PP12}
J.~Picado and A.~Pultr.
\newblock {\em Frames and locales: Topology without points}.
\newblock Frontiers in Mathematics. Birkh\"auser/Springer Basel AG, Basel,
  2012.

\bibitem[Smy92]{Smy92}
M.~B. Smyth.
\newblock Stable compactification. {I}.
\newblock {\em J. London Math. Soc. (2)}, 45(2):321--340, 1992.

\bibitem[Tow96]{Tow96}
C.~F. Townsend.
\newblock {\em Preframe techniques in constructive locale theory}.
\newblock PhD thesis, Imperial College London (University of London), 1996.

\bibitem[War74]{War75}
A.~J. Ward.
\newblock Representations of proximity lattices.
\newblock {\em Ann. Univ. Sci. Budapest. E\"{o}tv\"{o}s Sect. Math.},
  17:41--57, 1974.

\end{thebibliography}
\end{document}